\documentclass[11pt,oneside,reqno]{amsart}

\usepackage{amsmath,amssymb,amsthm,textcomp}
\usepackage{amsfonts,graphicx}
\usepackage[mathscr]{eucal}
\usepackage{color}
\usepackage{csquotes}
\usepackage[backend=bibtex,%
firstinits=true,%
doi=false,%
isbn=true,%
url=false,%
maxnames=99]{biblatex}%

\AtEveryBibitem{\clearfield{issn}}
\AtEveryCitekey{\clearfield{issn}}
\numberwithin{equation}{section}
\DeclareNameAlias{sortname}{last-first}
\theoremstyle{definition}
\usepackage{mathtools}
\numberwithin{equation}{section}

\newcommand{\ncom}{\newcommand}

\ncom{\beq}{\begin{equation}}
	\ncom{\eeq}{\end{equation}}
\ncom{\bea}{\begin{eqnarray*}}
	\ncom{\eea}{\end{eqnarray*}}
\ncom{\beqa}{\begin{eqnarray}}
	\ncom{\eeqa}{\end{eqnarray}}
\ncom{\nno}{\nonumber}
\ncom{\non}{\nonumber}
\ncom{\ds}{\displaystyle}
\ncom{\half}{\frac{1}{2}}
\ncom{\mbx}{\makebox{.25cm}}
\ncom{\hs}{\mbox{\hspace{.25cm}}}
\ncom{\rar}{\rightarrow}
\ncom{\Rar}{\Rightarrow}
\ncom{\noin}{\noindent}
\ncom{\bc}{\begin{center}}
	\ncom{\ec}{\end{center}}
\ncom{\sz}{\scriptsize}
\ncom{\rf}{\ref}
\ncom{\s}{\sqrt{2}}
\ncom{\sgm}{\sigma}
\ncom{\Sgm}{\Sigma}
\ncom{\psgm}{\sigma^{\prime}}
\ncom{\dt}{\delta}
\ncom{\Dt}{\Delta}
\ncom{\lmd}{\lambda}
\ncom{\Lmd}{\Lambda}
\ncom{\Th}{\Theta}
\ncom{\e}{\eta}
\ncom{\eps}{\epsilon}
\ncom{\pcc}{\stackrel{P}{>}}
\ncom{\lp}{\stackrel{L_{p}}{>}}
\ncom{\dist}{{\rm\,dist}}
\ncom{\sspan}{{\rm\,span}}
\ncom{\re}{{\rm Re\,}}
\ncom{\im}{{\rm Im\,}}
\ncom{\sgn}{{\rm sgn\,}}
\ncom{\ba}{\begin{array}}
	\ncom{\ea}{\end{array}}
\ncom{\hone}{\mbox{\hspace{1em}}}
\ncom{\htwo}{\mbox{\hspace{2em}}}
\ncom{\hthree}{\mbox{\hspace{3em}}}
\ncom{\hfour}{\mbox{\hspace{4em}}}
\ncom{\vone}{\vskip 2ex}
\ncom{\vtwo}{\vskip 4ex}
\ncom{\vonee}{\vskip 1.5ex}
\ncom{\vthree}{\vskip 6ex}
\ncom{\vfour}{\vspace*{8ex}}
\ncom{\norm}{\|\;\;\|}
\ncom{\integ}[4]{\int_{#1}^{#2}\,{#3}\,d{#4}}
\ncom{\vspan}[1]{{{\rm\,span}\{ #1 \}}}
\ncom{\dm}[1]{ {\displaystyle{#1} } }
\ncom{\ri}[1]{{#1} \index{#1}}

\newtheorem{theorem}{\bf Theorem}[section]
\newtheorem{remark}{\bf Remark}[section]
\newtheorem{proposition}{Proposition}[section]

\newtheoremstyle
{remarkstyle}
{}
{11pt}
{}
{}
{\bfseries}
{:}
{     }
{\thmname{#1} \thmnumber{#2} }

\theoremstyle{remarkstyle}

\def\r{\rangle}

\def\eps{\varepsilon}

\begin{document}
	\title{On the Superposition and Thinning of Generalized Counting Processes}
	\author[Manisha Dhillon]{Manisha Dhillon}
	\address{Manisha Dhillon, Department of Mathematics, Indian Institute of Technology Bhilai, Durg 491002, India.} \email{manishadh@iitbhilai.ac.in}
	\author[Kuldeep Kumar Kataria]{Kuldeep Kumar Kataria}
	\address{Kuldeep Kumar Kataria, Department of Mathematics, Indian Institute of Technology Bhilai, Durg 491002, India.} \email{kuldeepk@iitbhilai.ac.in}
	
	\date{\today}
	\begin{abstract}
		In this paper, we study the merging and splitting of generalized counting processes (GCPs). First, we study the merging of a finite number of independent GCPs and then extend it to the case of countably infinite. The merged process is observed to be a GCP with increased arrival rates. It is shown that a packet of jumps arrives in the merged process according to the Poisson process. Also, we study two different types of splitting of a GCP. In the first type, we study the splitting of jumps of a GCP where the probability of simultaneous jumps in the split components is negligible. In the second type, we consider the splitting of jumps in which there is a possibility of simultaneous jumps in the split components.  It is shown that the split components are GCPs with certain decreased jump rates. Moreover, the independence of split components is established. Later, we discuss applications of the obtained results to industrial fishing problem and hotel booking management system.
	\end{abstract}
\maketitle
\section{Introduction}
Recently, there has been a considerable interest in studying the generalizations of Poisson process. The generalized counting process (GCP) is one among its generalizations which was first introduced and studied by Di Crescenzo {\it et} {\it al}. (2016). It performs $k$ kinds of jumps of size $1,2,\dots,k$ with positive rates $\lambda_1,\lambda_2,\dots,\lambda_k$, respectively. We denote it by $\{M(t)\}_{t\geq0}$. In an infinitesimal interval of length $h$, its jumps probabilities are given by 
\begin{equation*}
	\mathrm{Pr}\{M(t+h)=n+j\big|M(t)=n\}=\begin{cases}
		1-\sum_{j=1}^{k}\lambda_jh+o(h), \ j=0,\vspace{.1cm}\\
		\lambda_j h+o(h), \  1\leq j\leq k,\vspace{.1cm}\\
		o(h), \ j\geq k+1,
	\end{cases}
\end{equation*}	
where $o(h)/h\to0 $ as $h\to0$.
For each $n\geq0$, the state probability of GCP is given by (see Di Crescenzo {\it et al.} (2016))
\begin{equation}\label{gcpnt}
	p(n,t)=\sum_{\Omega(k,n)}\prod_{j=1}^{k}\frac{(\lambda_jt)^{x_j}}{x_j!}e^{-\lambda_j t},
\end{equation}
where $\Omega(k,n)=\{(x_1,x_2,\ldots,x_k):\sum_{j=1}^{k}jx_j=n,\ x_j\in \mathbb{N}_0\}$. Here, $\mathbb{N}_0$ denotes the set of non-negative integers. Also, its probability generating function (pgf) is given by
\begin{equation}\label{pgfgcp}
	G(u,t)=\mathbb{E}\left(u^{M(t)}\right)=\exp\bigg(-\sum_{j=1}^{k}\lambda_{j}(1-u^{j})t\bigg),\  \ |u|\le 1.
\end{equation}
Thus, its mean and variance are 
\begin{equation}\label{meanvar}
	\mathbb{E}(M(t))=\sum_{j=1}^{k}j\lambda_jt\, \, \,  \text{and}\, \, \, \operatorname{Var}(M(t))=\sum_{j=1}^{k}j^2\lambda_jt,
\end{equation}
respectively. 

Kataria and Khandakar $(2022)$ obtained a limiting and a martingale result for the GCP. Also, they established a recurrence relation for its probability mass function (pmf) and discussed an application of the GCP to risk theory. For $k=1$, the GCP reduces to the homogeneous Poisson process. Some other particular cases of the GCP are Poisson process of order $k$ (see Kostadinova and Minkova (2013)), P\'olya-Aeppli process, P\'olya-Aeppli process of order $k$ (see  Chukova and Minkova (2015)) and the negative binomial process (see Kozubowski and Podg\'orski (2009)) {\it etc}.  

It is well known that the merging of two independent Poisson processes is again a Poisson process whose arrival rate is equal to the sum of the arrival rates of merging components. Also, the split components of a Poisson process are again Poisson with certain decreased jump rates. In this paper, we do a similar study for independent GCPs. The paper is organized as follows: 

In Section 2, the merging of independent GCPs is discussed. It is shown that the merged process is a GCP with increased jump rates. First, we consider the merging of finitely many independent GCPs and then extend the study to the case of countably many GCPs. In Section 3, the jump probabilities originating from a merging component are obtained. The cases of finite and countably many independent GCPs are treated separately. In Section $4$, we show that the arrival of a packet of jumps in the merged process is a Poisson process. This result implies that, in a GCP, a packet of jumps arrives according to the Poisson process. 
In Section 5, we study two types of the splitting of jumps in GCP. In both the cases, we study the splitting for two split components and then for $q\geq 2$ split components. Moreover, in the first type of splitting, we show that the split components are independent GCPs and in the second type of splitting, we establish that the split components are GCPs but not necessarily independent.
In the last section, we give applications of the merging and splitting of GCPs to industrial fishing problem and hotel booking management system.

\section{Merging of independent GCPs}\label{Sec1}
Here, we study the merging of finite number of independent GCPs. First, we consider the binary case. Let $ \{M_1(t)\}_{t\geq0}$ and $\{M_2(t)\}_{t\geq0}$ be two independent GCPs such that $\{M_1(t)\}_{t\geq0}$ performs jumps of amplitude $1, 2,\ldots,k_1$ with positive rates $\lambda_1,\lambda_2,\dots,\lambda_{k_1}$ and $\{M_2(t)\}_{t\geq0}$ performs jumps of amplitude $1, 2,\ldots,k_2$ with positive rates $\mu_1,\mu_2,\ldots,\mu_{k_2}$, respectively. The merging of $\{M_1(t)\}_{t\geq0}$ and $\{M_2(t)\}_{t\geq0}$ is a merged process denoted by $\{\mathscr{M}^2(t)\}_{t\geq0}$. It is defined as follows:
\begin{equation*}
	\mathscr{M}^2(t)\coloneqq M_1(t)+M_2(t).
\end{equation*}
In an infinitesimal interval of length $h$, a packet of jumps registers in the merged process  $\{\mathscr{M}^2(t)\}_{t\geq0}$ with some positive probability if and only if the jumps are observed in exactly one of the merging components. It turns out that the probability of simultaneous jumps in both the merging processes is of order $o(h)$. 

The joint probabilities of merging components $\{M_1(t)\}_{t\geq0}$ and $\{M_2(t)\}_{t\geq0}$ in an infinitesimal interval of length $h$ are demonstrated in Figure \ref{fig1}. These probabilities are positive in the shaded region and negligible, that is, of order $o(h)$ in the unshaded region. 
\begin{figure}[htp]
	\includegraphics[width=13cm]{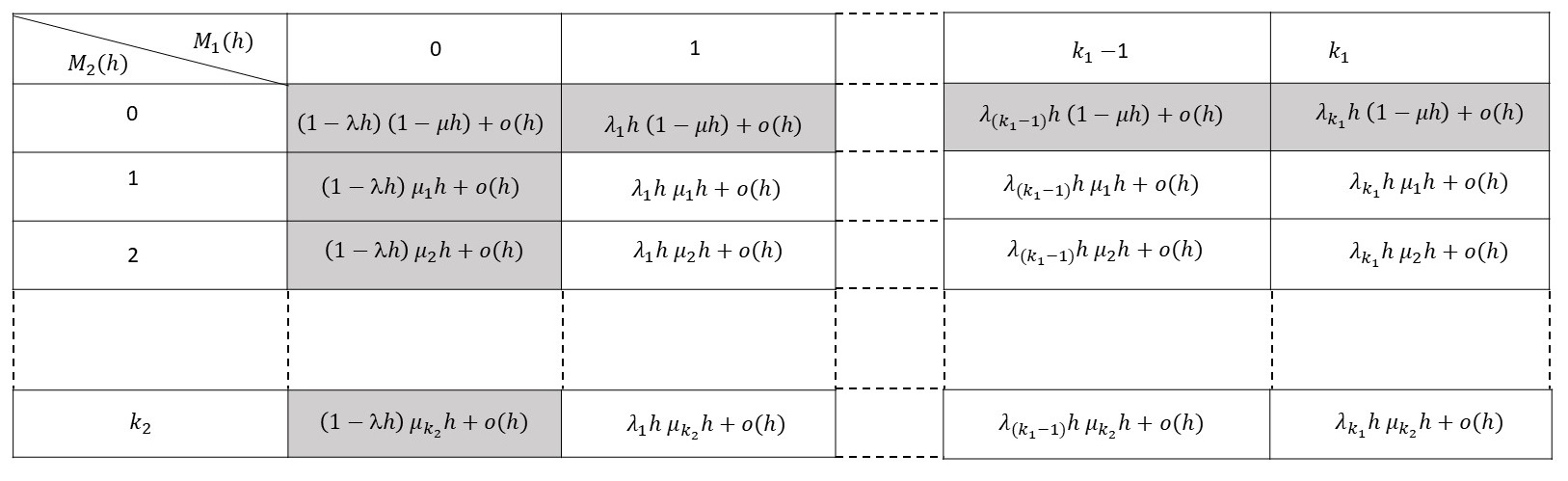}
	\vspace{-.3cm}\caption{Joint probabilities of the merging components in an infinitesimal interval of length $h$.}\label{fig1}
\end{figure}

Let $k=\max\{k_1,k_2\}$. It's evident  that the merged process $\{\mathscr{M}^2(t)\}_{t\geq0}$  performs $k$ kinds of jumps of amplitude $1,2,\dots,k$. Also, the merged process $\{\mathscr{M}^2(t)\}_{t\geq0}$ has independent and stationary increments because the merging components $\{M_1(t)\}_{t\geq0}$ and $\{M_2(t)\}_{t\geq0}$ are independent and have independent and stationary increments. This is shown in the Proposition $\ref{mm}$. We will use the following notation: For a counting process $\{X(t)\}_{t\geq0}$, let	$X(t_1,t_2]\coloneqq X(t_2)-X(t_1), \ t_2>t_1\geq0$.
\begin{proposition}\label{mm}
	The merged process $\{\mathscr{M}^2(t)\}_{t\geq0}$ has independent and stationary increments.
\end{proposition}
\begin{proof}

	\normalsize	
	Let $0\leq t_0<t_1<t_2<\dots<t_m<\infty$. Then, \scriptsize
	{\begin{align}
			\mathrm{Pr}&\{\mathscr{M}^2(t_0,t_1]=n_1,\mathscr{M}^2(t_1,t_2]=n_2,\dots,\mathscr{M}^2(t_{m-1},t_m]=n_m \}\nonumber\\
			&=\sum_{l_1,l_2,\dots,l_m}\mathrm{Pr}\big\{M_1(t_0,t_1]=l_1,M_2(t_0,t_1]=n_1-l_1,\dots,M_1(t_{m-1},t_m]=l_m,M_2(t_{m-1},t_m]=n_m-l_m\big\}\nonumber\\
			&=\sum_{l_1,l_2,\dots,l_m}\mathrm{Pr}\left(\cap_{i=1}^{m}\left\{M_1(t_{i-1},t_i]=l_i\right\}\right)\mathrm{Pr}\left(\cap_{i=1}^{m}\left\{M_2(t_{i-1},t_i]=n_i-l_i\right\}\right)\, \label{indof12}\\
			&=\sum_{l_1,l_2,\dots,l_m}\prod_{i=1}^{m}\mathrm{Pr}\left\{M_1(t_{i-1},t_i]=l_i\right\}\mathrm{Pr}\left\{M_2(t_{i-1},t_i]=n_i-l_i\right\}\label{incre12}\\
			&=\prod_{i=1}^{m}\sum_{l_1,l_2,\dots,l_m}\mathrm{Pr}\big\{M_1(t_{i-1},t_i]=l_i\big\}\mathrm{Pr}\left\{M_2(t_{i-1},t_i]=n_i-l_i\right\}\nonumber\\
			&=\prod_{i=1}^{m}\mathrm{Pr}\left\{\mathscr{M}^2(t_{i-1},t_i]=n_i\right\},\nonumber
	\end{align}}
	\normalsize	where we have used the independence of $\{M_1(t)\}_{t\geq0}$ and $\{M_2(t)\}_{t\geq0}$ and their independent increments to obtain \eqref{indof12} and \eqref{incre12}, respectively. Thus, the merged process $\{\mathscr{M}^2(t)\}_{t\geq0}$ has independent increments. Next, we show that it has stationary increments. For $s>0$, we have
	\begin{align*}
		\mathrm{Pr}\{\mathscr{M}^2(t,t+s]=m\}&=\sum_{l=1}^{m}\mathrm{Pr}\{M_1(t,t+s]=l\}\mathrm{Pr}\{M_2(t,t+s]=m-l\}\\
		&=\sum_{l=1}^{m}\mathrm{Pr}\{M_1(s)=l\}\mathrm{Pr}\{M_2(s)=m-l\}\\
		&=	\mathrm{Pr}\{\mathscr{M}^2(s)=m\}.
	\end{align*}
	The above steps follows from the independence of $\{M_1(t)\}_{t\geq0}$ and $\{M_2(t)\}_{t\geq0}$ and their stationary increments.
	This completes the proof.
\end{proof}
If $\{M_1(t)\}_{t\geq0}$ performs no jump in $(y,y+h]$, then 
\begin{equation*}
	\mathscr{M}^2(y,y+h]=\mathscr{M}^2(y+h)-\mathscr{M}^2(y)=M_2(y+h)-M_2(y).
\end{equation*} 
Similarly, 
\begin{equation*}
	\mathscr{M}^2(y,y+h]=\mathscr{M}^2(y+h)-\mathscr{M}^2(y)=M_1(y+h)-M_1(y),
\end{equation*} 
provided $\{M_2(t)\}_{t\geq0}$ performs no jump in $(y,y+h]$. As $\{\mathscr{M}^2(t)\}_{t\geq0}$ has stationary increments, we have
\begin{align*}
	\mathscr{M}^2(t,t+h]&\overset{d}{=}\mathscr{M}^2(h),
\end{align*} 
where $\overset{d}{=}$ denotes the equality in distribution.

First, we consider the case when $k_1<k_2$. The jump probabilities of $\{\mathscr{M}^2(t)\}_{t\geq0}$ in an infinitesimal interval of length $h$ are obtained as follows: For $j=0$, we have  
\begin{align*}
	\mathrm{Pr}\{\mathscr{M}^2(h)=0\}&=		\mathrm{Pr}\{M_1(h)=0\}\mathrm{Pr}\{M_2(h)=0\}+o(h)\\
	&=1-(\lambda +\mu)h+o(h),
\end{align*}
where $\lambda=\lambda_1+\lambda_2+\dots+\lambda_{k_1}$ and $\mu=\mu_1+\mu_2+\dots+\mu_{k_2}$. 

For $1\leq j\leq k_1$, we have
\begin{align*}
	\mathrm{Pr}\{\mathscr{M}^2(h)=j\}&=		\mathrm{Pr}\{M_1(h)=j\}\mathrm{Pr}\{M_2(h)=0\}+\mathrm{Pr}\{M_1(h)=0\}\mathrm{Pr}\{M_2(h)=j\}+o(h)\\
	&=(\lambda_j+\mu_j)h+o(h),
\end{align*}
for $ k_1<j\leq k_2$, we have \begin{equation*}
	\mathrm{Pr}\{\mathscr{M}^2(h)=j\}=\mathrm{Pr}\{M_2(h)=j\} =\mu_jh+o(h)
\end{equation*}
and for $j>k_2$, we have 
\begin{equation*}
	\mathrm{Pr}\{\mathscr{M}^2(h)=j\}=o(h).
\end{equation*}
So, for any $i\geq 0$, we get
\begin{equation}\label{stprob}
	\mathrm{Pr}\{\mathscr{M}^2(t+h)=i+j\big|\mathscr{M}^2(t)=i\}	=\begin{cases}
		1-(\lambda +\mu)h+o(h),\ j=0,\\
		(\lambda_j+\mu_j)h+o(h),\ 1\leq j\leq k_1,\\
		\mu_jh+o(h),\ k_1<j\leq k_2,\\
		o(h), \ \text{otherwise},
	\end{cases}
\end{equation}
where $o(h)/h\to 0$ as $ h\to 0$.

Hence, the merged process $\{\mathscr{M}^2(t)\}_{t\geq0}$  is a GCP which performs $k_2$ kinds of jumps of amplitude $1, 2, \ldots, k_2$ with positive rates $\lambda_1+\mu_1$, $\lambda_2+\mu_2$, $\dots$, $\lambda_{k_1}+\mu_{k_1}$, $\mu_{k_1+1}$, $\ldots$, $\mu_{k_2}$, respectively.
%

\begin{remark}
	Similarly, for $k_2<k_1$, the merged process $\{\mathscr{M}^2(t)\}_{t\geq0}$ is a GCP which  performs $k_1$ kinds of jumps of amplitude $1, 2, \dots, k_1$ with positive rates $\lambda_1+\mu_1$, $\lambda_2+\mu_2$, $\dots$, $\lambda_{k_2}+\mu_{k_2}$, $\lambda_{k_2+1}$, $\dots$, $\lambda_{k_1}$, respectively. 
\end{remark}
\begin{remark}
	For $k_1=k_2=k $ (say), the merged process  $\{\mathscr{M}^2(t)\}_{t\geq0}$ is a GCP with increased rates $\beta_j=\lambda_j+\mu_j,\,  1\leq j\leq k$. 
\end{remark}
It is important to note that, for $k_1=k_2=1$, the above result reduces to the merging of two independent Poisson processes (see Bertsekas and Tsitsiklis $(2008)$, p. $24$).

%

\subsection{Merging of finitely many independent GCPs}\label{section 1.2}
Let us consider $q$ independent GCPs  $\{M_1(t)\}_{t\geq0},\{M_2(t)\}_{t\geq0},\dots, \{M_q(t)\}_{t\geq0}$, such that for each $1\leq i\leq q$ the GCP $\{M_i(t)\}_{t\geq0}$ performs $k_i$ kinds of jumps of amplitude $1,2,\dots,k_i$  with positive rates $\lambda^{(i)}_1$, $\lambda^{(i)}_2$, $\dots$, $\lambda^{(i)}_{k_i}$, respectively. The merging of these GCPs results in a merged process $\{\mathscr{M}^q(t)\}_{t\geq0}$ defined as follows:
\begin{equation*}
	\mathscr{M}^q(t)\coloneqq M_1(t)+M_2(t)+\dots+M_q(t).
\end{equation*}  

In an infinitesimal interval of length $h$, a packet of jumps is registered in the merged process $\{\mathscr{M}^q(t)\}_{t\geq0}$ with some positive probability if and only if the jumps are registered in exactly one of the merging components. Moreover, the probability of simultaneous occurrence of jumps in more than one GCPs is of order $o(h)$. The above observations is similar to the case of merging of two independent GCPs.
The merging of $q$ independent GCPs is depicted in Figure $\ref{mergenn}$.
Without loss of generality, we can take $k_1\leq k_2\leq\dots\leq k_q$. Then, as in the case of merging of two independent GCPs, the merged process $\{\mathscr{M}^q(t)\}_{t\geq0}$ has independent and stationary increments, and it performs $k_q$ kinds of jumps of amplitude $1,2,\dots,k_{q}$. Further, 
\begin{equation*}
	\mathscr{M}^q(t,t+h]=\mathscr{M}^q(t+h)-\mathscr{M}^q(t)=M_i(t+h)-M_i(t),
\end{equation*} 
provided $M_j(t+h)-M_j(t)=0$, for all $i\neq j$ and $1\leq i,j\leq q$.

Let
\begin{align*}
	k_{(1)}&=\min\{k_1,k_2,\dots,k_q\},	\\
	k_{(2)}&=\min\{k_1,k_2,\dots,k_q\}\backslash\{k_{(1)}\},\\
	&\hspace*{.2cm}\vdots\\
	k_{(l-1)}&=\max\{k_1,k_2,\dots,k_q\}\backslash\{k_{(l)}\},\\
	k_{(l)}&=\max\{k_1,k_2,\dots,k_q\}.
\end{align*}
So, among these $q$ independent GCPs there are a total of $d_i$ number of GCPs that perform $k_{(i)}$ kinds of jumps such that $d_1+d_2+\dots+d_l=q$. As $\{\mathscr{M}^q(t)\}_{t\geq0}$ has stationary increments, we have
\begin{equation*}
	\mathscr{M}^q(t,t+h]\overset{d}{=}\mathscr{M}^q(h).
\end{equation*} 
So,
\begin{equation}\label{jumppmf}
	\mathrm{Pr}\{\mathscr{M}^q(t,t+h]=i+j|\mathscr{M}^q(t)=i\}=\begin{cases}
		1-\beta h+o(h), \, j=0,\\
		\beta_j h+o(h),\ 1\leq j\leq k_{(l)},\\
		o(h),\, \text{otherwise},
	\end{cases}
\end{equation}
where $o(h)/h\to 0$ as $h\to 0$ and
\begin{equation}\label{betaqjs}
	\beta_j=\begin{cases}
		\sum_{i=1}^{q}\lambda^{(i)}_j, \ 1\leq j\leq k_{(1)}, 
		\vspace*{.1 cm}\\
		\sum_{i=r_1}^{q}\lambda^{(i)}_j,\ k_{(1)} +1\leq j\leq k_{(2)},
		\vspace*{.1 cm}\\
		\hspace*{.8cm}\vdots\\
		\sum_{i=r_{l-2}}^{q}\lambda^{(i)}_j, \ k_{(l-2)}+1\leq j\leq k_{(l-1)},\vspace*{.1cm}\\
		\sum_{i=r_{l-1}}^{q}\lambda^{(i)}_j, \ k_{(l-1)}+1\leq j\leq k_{(l)}.
	\end{cases}
\end{equation}
Here, $\beta = \beta_1+\beta_2+\dots+\beta_{k_{(l)}}$ and
$r_m=\sum_{i=1}^{m}d_i+1$.
Thus, the merged process $\{\mathscr{M}^q(t)\}_{t\geq0}$ performs jumps of amplitude $1,2,\dots, k_{(l)}$ with positive rates $\beta_1,\beta_2,\dots,\beta_{k_{(l)}}$, respectively.
\begin{figure}[htp]
	\includegraphics[width=14cm]{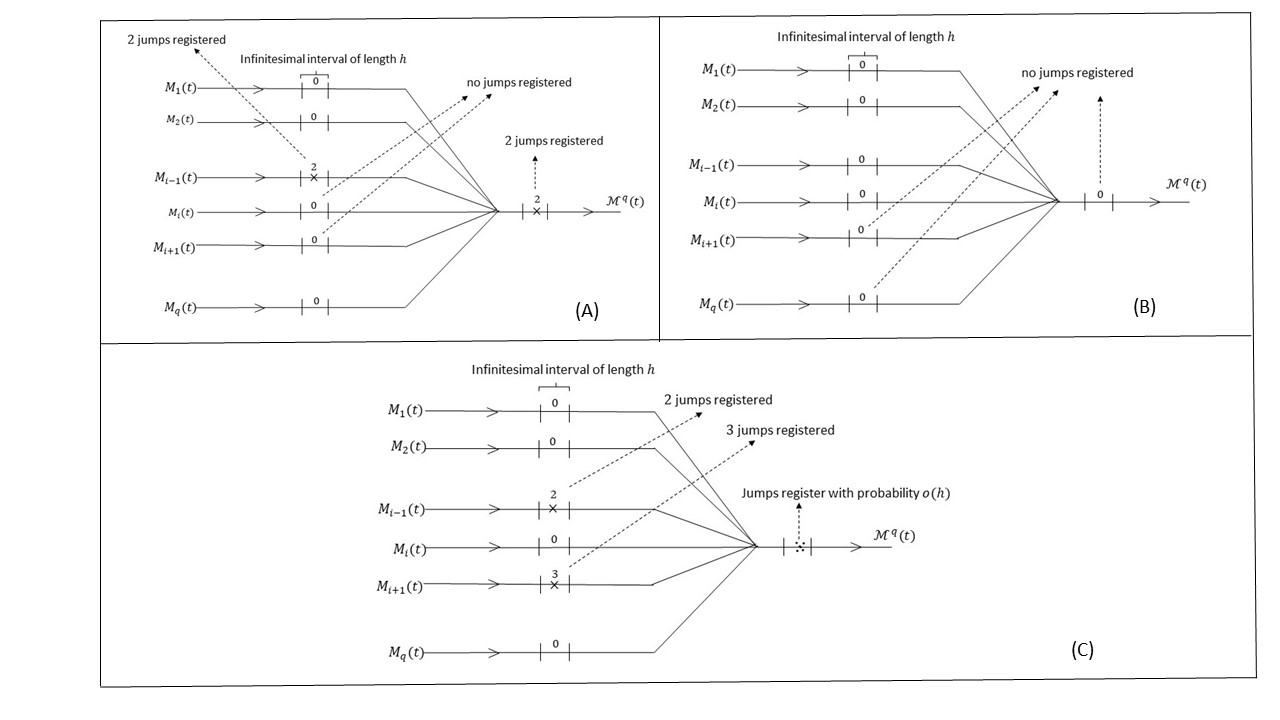}
	\caption{Merging of $q$ independent GCPs.}\label{mergenn}
\end{figure}
\begin{remark}
	If $k_1<k_2<\dots<k_q$ then $k_{(i)}=k_i$ and $l=q$. In this case, the merged process $\{\mathscr{M}^q(t)\}_{t\geq0}$ is a GCP that performs $1,2,\dots,k_q$ kinds of jumps with positive rates $\beta_1=\sum_{i=1}^{q}\lambda_1^{(i)}$, $\beta_2=\sum_{i=2}^{q}\lambda_2^{(i)}$, $\dots$, $\beta_{k_q}=\lambda_{k_q}^{(q)}$, respectively.
\end{remark}
\begin{remark}
	If $k_1=k_2=\dots=k_q=k $ (say) then $l=1$. In this case too, the merged process $\{\mathscr{M}^q(t)\}_{t\geq0}$ is a GCP. 
	Alternatively, this can be established from its pgf which can be obtained as follows:
	\begin{align*}
		\mathbb{E}\left(u^{\mathscr{M}^q(t)}\right)&=\prod_{i=1}^{q}\exp\left(-\sum_{j=1}^{k}\lambda_j^{(i)}t(1-u^j)\right),\, \text{(as $\{M_i(t)\}_{t\geq0}$ are independent GCPs)}\\
		&=\exp\left(-\sum_{j=1}^{k}\sum_{i=1}^{q}\lambda_j^{(i)}t(1-u^j)\right).
	\end{align*}
	This shows that the merged process $\{\mathscr{M}^q(t)\}_{t\geq0}$ is a GCP that performs $1,2,\dots,k$ kinds of jumps with positive rates $\sum_{i=1}^{q}\lambda_1^{(i)},\sum_{i=1}^{q}\lambda_2^{(i)},\dots,\sum_{i=1}^{q}\lambda_k^{(i)}$, respectively.
\end{remark}

\subsection{Merging of countably many independent GCPs}\label{Section 1.3}
Let us consider a collection of countably infinite independent GCPs $\{M_1(t)\}_{t\geq0}$, $\{M_2(t)\}_{t\geq0}$,\ldots, such that for each  $i\geq 1$ the GCP $\{M_i(t)\}_{t\geq0}$  performs jumps of amplitude $1,2,\ldots,k_i$ with positive rates $\lambda^{(i)}_1,\lambda^{(i)}_2,\dots,\lambda^{(i)}_{k_i}$, respectively. Assume that $\lambda_j^{(i)}=0$ for all $j>k_i$. Without loss of generality, we consider $k_1\leq k_2\leq k_3\dots$.\\ 
Let
\begin{align*}
	k_{(1)}&=\min\{k_1,k_2,\ldots\},\\
	k_{(2)}&=\min\{k_1,k_2,\dots\}\backslash \{k_{(1)}\},\\
	&\hspace{.2cm}\vdots\\
	k_{(q-1)}&=\max\{k_1,k_2,\dots\}\backslash \{k_{(q)}\},\\
	k_{(q)}&=\max\{k_1,k_2,\dots\},
\end{align*} 
and assume $k_{(q)}<\infty$. 
\begin{theorem}
	For all $ 1\leq j\leq k_{(q)}$, if
	\begin{equation}\label{countrate}
		\beta_j=\sum_{i=1}^{\infty}\lambda_j^{(i)}<\infty,
	\end{equation}
	then for any fix $t\geq0$, the merged process   $\mathscr{M}^\infty(t):=\sum_{i=1}^{\infty}M_i(t)$
	converges with probability $1$. Moreover, $\{\mathscr{M}^\infty(t)\}_{t\geq0}$ is a GCP which performs $k_{(q)}$ kinds of jumps with amplitude $1,2,\dots,k_{(q)}$ and corresponding rates $\beta_1$, $\beta_2$, $\dots$, $\beta_{k_{(q)}}$. If for any $1\leq j\leq k_{(q)}$, $\beta_j$ diverges then $\{\mathscr{M}^\infty(t)\}_{t
		\geq0}$ diverges with probability $1$. 
\end{theorem}
\begin{proof}
	Recall from Section $\ref{section 1.2}$ that  $\mathscr{M}^n(t)=\sum_{i=1}^{n}M_i(t)$ 
	is a GCP which performs $1,2,\dots,k_n$ kinds of jumps with positive rates $\beta_1^n=\sum_{i=1}^{n}\lambda_1^{(i)}$,  $\beta_2^n=\sum_{i=1}^{n}\lambda_2^{(i)}$, $\dots$,  $\beta_{k_n}^n=\sum_{i=1}^{n}\lambda_{k_n}^{(i)}$, respectively.
	For any fix $r\geq0$, we have
	\begin{align}
		\mathrm{Pr}\{\mathscr{M}^\infty(t)\leq r\}&=\lim_{n\to\infty}\mathrm{Pr}\{\mathscr{M}^n(t)\leq r\}\nonumber\\
		&=\lim_{n\to\infty}\sum_{m=0}^{r}\sum_{\Omega(k_n,m)}\prod_{j=1}^{k_n}\frac{(\beta_j^nt)^{x_j}}{x_j!}e^{-\beta_j^nt},\, \text{(using \eqref{gcpnt})} \label{betajconven}\\
		&=\sum_{m=0}^{r}\sum_{\Omega(k_{(q)},m)}\prod_{j=1}^{k_{(q)}}\lim_{n\to\infty}\frac{(\beta_j^nt)^{x_j}}{x_j!}e^{-\beta_j^nt},\, \text{($k_n\to k_{(q)}$ as $n\to \infty$) }.\nonumber
	\end{align}

	If $\beta_j^n$ converges to $\beta_j$ for all $j$ then
	\begin{equation*}
		\mathrm{Pr}\{\mathscr{M}^\infty(t)\leq r\}=\sum_{m=0}^{r}\sum_{\Omega\left(k_{(q)},m\right)}\prod_{j=1}^{k_{(q)}}\frac{(\beta_jt)^{x_j}}{x_j!}e^{-\beta_j t}.
	\end{equation*}
	So,
	\begin{equation*}
		\mathrm{Pr}\{\mathscr{M}^\infty(t)= r\}=\mathrm{Pr}\{\mathscr{M}^\infty(t)\leq r\}-\mathrm{Pr}\{\mathscr{M}^\infty(t)\leq r-1\}=\sum_{\Omega\left(k_{(q)},r\right)}\prod_{j=1}^{k_{(q)}}\frac{(\beta_jt)^{x_j}}{x_j!}e^{-\beta_j t}.
	\end{equation*}
	Thus, $\{\mathscr{M}^\infty(t)\}_{t\geq0}$ is a GCP which performs jumps of size $1,2,\dots,k_{(q)}$ with corresponding rates $\beta_1$, $\beta_2$, $\dots$, $\beta_{k_{(q)}}$. 
	
	If for any $1\leq j\leq k_{(q)}$, $\beta_j^n$ diverges then from \eqref{betajconven}, we have 
	\begin{equation*}
		\mathrm{Pr}\{\mathscr{M}^\infty(t)\leq r\}=\lim_{n\to\infty}e^{-\left(\beta_1^n+\beta_2^n+\dots+\beta_{k_n}^n\right)t}\sum_{m=0}^{r}\sum_{\Omega(k_n,m)}\prod_{j=1}^{k_n}\frac{(\beta_j^nt)^{x_j}}{x_j!}\to 0,
	\end{equation*}
	which implies $\mathrm{Pr}\{\mathscr{M}^\infty(t)> r\}=1$ and it holds for all $r\geq0$. Therefore, $\{\mathscr{M}^\infty(t)\}_{t\geq0}$ diverges with probability $1$. This completes the proof.
\end{proof}
In an infinitesimal interval of length $h$, a packet of jumps registers in the merged process $\{\mathscr{M}^\infty(t)\}_{t\geq 0}$ if and only if this packet of jumps is observed in exactly one of the merging components. The probability of simultaneous jumps in more than one merging components is of order $o(h)$. Similar to the case of merging of finitely many independent GCPs, here also the merging components are independent, and have independent and stationary increments. So, the merged process $\{\mathscr{M}^\infty(t)\}_{t\geq0}$ has independent and stationary increments. For any $i\geq0$, its jump probabilities are given by
\begin{equation}\label{stcde}
	\mathrm{Pr}\{\mathscr{M}^\infty(t+h)=i+j|\mathscr{M}^\infty(t)=i\}=\begin{cases}
		1-\beta h+o(h), \, j=0,\\
		\beta_jh+o(h), \, 1\le j\le k_{(q)},\\
		o(h), \, \text{otherwise},
	\end{cases}
\end{equation} 
where $\beta=\beta_1+\beta_2+\dots+\beta_{k_{(q)}}$ and  $\beta_j$'s are given by \eqref{countrate}.

\begin{remark}
	For $k_1=k_2=k_3=\dots=k$, the pgf of the merged process $\{\mathscr{M}^\infty(t)\}_{t\geq0}$ is obtained as follows:
	\begin{align*}
		\mathbb{E}\left(u^{\mathscr{M}^\infty(t)}\right)
		&=\mathbb{E}\left(\lim_{n\to\infty}\prod_{i=1}^{n}u^{M_i(t)}\right)\\
		&=\lim_{n\to\infty}\prod_{i=1}^{n}\mathbb{E}\left(u^{M_i(t)}\right),\, \text{(as $\{M_i(t)\}_{t\geq0}$ are independent)}\\
		&=\exp\left(-\sum_{j=1}^{k}\sum_{i=1}^{\infty}\lambda_j^{(i)}t(1-u^j)\right),
	\end{align*}
	where the last step follows on using \eqref{pgfgcp}.
	Thus, the merged  $\{\mathscr{M}^\infty(t)\}_{t\geq0}$ is a GCP provided $\sum_{i=1}^{\infty}\lambda_j^{(i)}<\infty$ for all $1\leq j\leq k$. It performs $k$ kinds of jumps of amplitude $1,2,\dots,k$ with positive rates $\sum_{i=1}^{\infty}\lambda_1^{(i)},\sum_{i=1}^{\infty}\lambda_2^{(i)},\dots$, $\sum_{i=1}^{\infty}\lambda_k^{(i)}$, respectively.
	
	For $k=1$, this reduces to the merging of countable many independent Poisson processes (see Kingman (1993), p. 5).
\end{remark}

\section{Jump probabilities from a particular process}
Let $\{M_1(t)\}_{t\geq0}$ and $\{M_2(t)\}_{t\geq0}$ be two independent GCPs that perform $k_1$ and $k_2$ kinds of jumps, respectively. In Section \ref{Sec1}, it is shown that the merged process $\{\mathscr{M}^2(t)\}_{t\geq0}$ is a GCP that performs $k=\max\{k_1,k_2\}$ kinds of jumps. For $k_1<k_2$, its jump probabilities are given by \eqref{stprob}.

\begin{proposition}\label{pr1thm}
	Let $k_1<k_2$.
	If a packet of $j$ size of jumps is registered in the merged process $\{\mathscr{M}^2(t)\}_{t\geq0}$ then the probability that it originates from $\{M_1(t)\}_{t\geq0}$ is given by
	\begin{equation*}
		\mathrm{Pr}\{M_1(h)=j|\mathscr{M}^2(h)=j\}=
		\begin{cases}
			\frac{\lambda_j}{\lambda_j+\mu_j}, \ 1\leq j \leq k_1,\\
			0, \ k_1 +1\leq j\leq k_2,
		\end{cases}
	\end{equation*}
	and the probability that it originates from $\{M_2(t)\}_{t\geq0}$ is given by
	\begin{equation*}
		\mathrm{Pr}\{M_2(h)=j | \mathscr{M}^2(h)=j\}=
		\begin{cases}
			\frac{\mu_j}{\lambda_j+\mu_j}, \ 1\leq j \leq k_1,\\
			1, \ k_1 +1\leq j\leq k_2,
		\end{cases}
	\end{equation*}
	provided $o(h)/h\to 0$ as $h\to 0$.
\end{proposition}
\begin{proof}
	Let us consider an infinitesimal interval of length $h$, such that $o(h)/h\to0$ as $h\to 0$. Then,
	\begin{equation*}
		\mathrm{Pr}\{M_1(h)=j|\mathscr{M}^2(h)=j\}=
		\frac{\mathrm{Pr}\{M_1(h)=j,\mathscr{M}^2(h)=j\}}{\mathrm{Pr}\{\mathscr{M}^2(h)=j\}}\\
		=\begin{cases}
			\frac{\lambda_j}{\lambda_j+\mu_j}, \ 1\leq j \leq k_1,\,\vspace*{.1cm}\\
			0, \ k_1 +1\leq j\leq k_2,
		\end{cases}
	\end{equation*}
	which follows because the GCP $\{M_1(t)\}_{t\geq0}$ performs a maximum of $k_1$ kinds of jumps. If more than $k_1$ kinds of jumps are registered in  $\{\mathscr{M}^2(t)\}_{t\geq0}$ then it is certain that those jumps are from  $\{M_2(t)\}_{t\geq0}$.  Similarly,	 
	\begin{equation*}
		\mathrm{Pr}\{M_2(h)=j|\mathscr{M}^2(h)=j\}
		=\frac{\mathrm{Pr}\{M_2(h)=j, \mathscr{M}^2(h)=j\}}{\mathrm{Pr}\{\mathscr{M}^2(h)=j\}}\\
		=\begin{cases}
			\frac{\mu_j}{\lambda_j+\mu_j}, \ 1\leq j \leq k_1,\vspace*{.1cm}\\
			1, \ k_1 +1\leq j\leq k_2.
		\end{cases}
	\end{equation*}
	This completes the proof.
\end{proof}
\begin{remark}
	For $k_2<k_1$,  
	we have
	\begin{equation*}
		\mathrm{Pr}\{M_1(h)=j|\mathscr{M}^2(h)=j\}=
		\begin{cases}
			\frac{\lambda_j}{\lambda_j+\mu_j}, \ 1\leq j \leq k_2,\\
			1, \ k_2 +1\leq j\leq k_1,
		\end{cases}
	\end{equation*}
	and 
	\begin{equation*}
		\mathrm{Pr}\{M_2(h)=j|\mathscr{M}^2(h)=j\}=
		\begin{cases}
			\frac{\mu_j}{\lambda_j+\mu_j}, \ 1\leq j \leq k_2,\\
			0, \ k_2 +1\leq j\leq k_1,
		\end{cases}
	\end{equation*}
	provided $o(h)/h\to 0$ as $h\to 0$.
\end{remark}
\begin{remark}
	For $k_1=k_2=k$ (say), the jump probabilities of $\{\mathscr{M}^2(t)\}_{t\geq0}$ are given by
	\begin{equation*}
		\mathrm{Pr}\{\mathscr{M}^2(t+h)=i+j|\mathscr{M}^2(t)=i\}=
		\begin{cases}
			1-(\lambda+\mu)h+o(h),\ j=0,\\
			(\lambda_j+\mu_j)h+o(h), \ 1\leq j\leq k,\\
			o(h), \ \mathrm{otherwise},
		\end{cases}
	\end{equation*} 
	where $\lambda+\mu=\sum_{j=1}^{k}(\lambda_j+\mu_j)$. In this case, we have
	$\mathrm{Pr}\{M_1(h)=j|\mathscr{M}^2(h)=j\}=	{\lambda_j}/{(\lambda_j+\mu_j)}$
	and 
	$\mathrm{Pr}\{M_2(h)=j | \mathscr{M}^2(h)=j\}=	{\mu_j}/{(\lambda_j+\mu_j)}$ for any $ j \in\{1,2,\dots,k\}$.
\end{remark}

\subsection{The case of merging of finitely many GCPs} Let us consider $q$ independent GCPs $\{M_1(t)\}_{t\geq0}$, $\{M_2(t)\}_{t\geq0}$, $\dots$,  $\{M_q(t)\}_{t\geq0}$, such that, $\{M_i(t)\}_{t\geq0}$ for each $i\in\{1,2,\dots,q\}$ performs $k_i$ kinds of jumps of amplitude $1,2,\dots k_i$. Without loss of generality, we consider $k_1\leq k_2\leq \dots\leq k_q$. Then, from Section \ref{section 1.2}, the merged process $\{\mathscr{M}^q(t)\}_{t\geq 0}$ is  a GCP which performs $k_{(l)}=k_q=\max\{k_1,k_2,\dots,k_q\}$ kinds of jumps. Its jump probabilities are given by $\eqref{jumppmf}$.
\begin{proposition}
	Let $k_1\leq k_2\leq \dots\leq k_q$. If a packet of $j$ size of jumps is registered in $\{\mathscr{M}^q(t)\}_{t\geq0}$ then the probability that it originates from $\{M_s(t)\}_{t\geq0}$ for $1\leq s<q$ is given by	
	\begin{equation*}
		\mathrm{Pr}\{M_s(h)=j|\mathscr{M}^q(h)=j\}=
		\begin{cases}
			\frac{\lambda^{(s)}_{j}}{\sum_{i=1}^{q}\lambda^{(i)}_j}, \ 1\leq j \leq k_{(1)},
			\vspace*{.2cm}\\
			\frac{\lambda^{(s)}_{j}}{\sum_{i=r_{p-1}}^{q}\lambda^{(i)}_j},\,k_{(p-1)}+1\leq j\leq k_{(p)}, \,p-1\leq \displaystyle \sum_{i=1}^{p-1}d_i\leq s-1,
			\vspace*{.2cm}\\
			0, \,k_{(p-1)}+1\leq j\leq k_{(p)},\, \displaystyle \sum_{i=1}^{p-1}d_i\geq s\, \text{or} \,  k_{(s)}+1\leq j\leq k_{(l)},
		\end{cases}
	\end{equation*}
	for any $p\in\{2,3,\dots,s\}$, and for $s=q$, we have 
	\begin{equation*}
		\mathrm{Pr}\{M_q(h)=j|\mathscr{M}^q(h)=j\}=
		\begin{cases}
			\frac{\lambda^{(q)}_{j}}{\sum_{i=1}^{q}\lambda^{(i)}_j}, \ 1\leq j \leq k_{(1)},
			\vspace*{.2cm}\\
			\frac{\lambda^{(q)}_{j}}{\sum_{i=r_1}^{q}\lambda^{(i)}_j}, \ k_{(1)}+1\leq j \leq k_{(2)},\\
			\ \ \ \ \ \ \vdots\\
			\frac{\lambda^{(q)}_{j}}{\sum_{i=r_{l-1}}^{q}\lambda^{(i)}_{j}}, \ k_{(l-1)} +1\leq j\leq k_{(l)}.
		\end{cases}
	\end{equation*}
	Here, $o(h)/h\to0$ as $h\to0$ and $r_m=\sum_{i=1}^{m}d_i+1$, where $d_i$ is the number of GCPs that perform $k_{(i)}$ kinds of jumps.
\end{proposition}
\begin{proof}
	Using $\eqref{jumppmf}$ and following the similar lines of the proof of Proposition $\ref{pr1thm}$, for $s=1$, we have
	\begin{equation*}
		\mathrm{Pr}\{M_1(h)=j | \mathscr{M}^q(h)=j\}=
		\begin{cases}
			\frac{\lambda^{(1)}_j}{\sum_{i=1}^{q}\lambda^{(i)}_{j}}, \ 1\leq j \leq k_{(1)},
			\vspace*{.1cm}\\
			0, \ k_{(1)} +1\leq j\leq k_{(l)}.
		\end{cases}
	\end{equation*}
	Similarly, for $s=2$ and $s=3$, we get
	\begin{equation*}
		\mathrm{Pr}\{M_2(h)=j | \mathscr{M}^q(h)=j\}=
		\begin{cases}
			\frac{\lambda^{(2)}_{j}}{\sum_{i=1}^{q}\lambda^{(i)}_j}, \ 1\leq j \leq k_{(1)},
			\vspace*{.2cm}\\
			\frac{\lambda^{(2)}_{j}}{\sum_{i=2}^{q}\lambda^{(i)}_j}, \,k_{(1)}+1\leq j\leq k_{(2)},\, d_1=1,\\
			0, \,k_{(1)}+1\leq j\leq k_{(2)},\,d_1>1\, \text{or}\ k_{(2)}+1\leq j\leq k_{(l)},
		\end{cases}
	\end{equation*}
	and 
	\begin{equation*}
		\mathrm{Pr}\{M_3(h)=j | \mathscr{M}^q(h)=j\}=
		\begin{cases}
			\frac{\lambda^{(3)}_{j}}{\sum_{i=1}^{q}\lambda^{(i)}_j}, \ 1\leq j \leq k_{(1)},
			\vspace*{.2cm}\\
			\frac{\lambda^{(3)}_{j}}{\sum_{i=r_{1}}^{q}\lambda^{(i)}_j}, \,k_{(1)}+1\leq j\leq k_{(2)},\, 1\leq d_1\leq 2,
			\vspace*{.2cm}\\
			\frac{\lambda^{(3)}_{j}}{\sum_{i=r_{2}}^{q}\lambda^{(i)}_j}, \,k_{(2)}+1\leq j\leq k_{(3)},\,  d_1+d_2=2,
			\vspace*{.2cm}\\
			0, \,k_{(1)}+1\leq j\leq k_{(2)},\,d_1\geq3\, \text{or} \\ 
			\hspace{.4cm}  k_{(2)}+1\leq j\leq k_{(3)},\,  d_1+d_2\geq 3\, \text{or}\,\\ 
			\hspace{.4cm}  k_{(3)}+1\leq j\leq k_{(l)},
		\end{cases}
	\end{equation*}
	respectively. Proceeding inductively, we get the required result.
\end{proof}
\begin{remark}
	If $k_1<k_2<\dots<k_q$ then $k_{(i)}=k_i$ and $l=q$. Thus, for any $1\leq s\leq q$, we have
	\begin{equation*}
		\mathrm{Pr}\{M_s(h)=j|\mathscr{M}^q(h)=j\}=
		\begin{cases}
			\frac{\lambda^{(s)}_{j}}{\sum_{i=1}^{q}\lambda^{(i)}_j}, \ 1\leq j \leq k_1,\vspace{.2cm}\\
			\frac{\lambda^{(s)}_{j}}{\sum_{i=2}^{q}\lambda^{(i)}_j}, \ k_1+1\leq j \leq k_2,\\
			\ \ \ \ \ \vdots\\
			\frac{\lambda^{(s)}_{j}}{\sum_{i=s}^{q}\lambda^{(i)}_j}, \ k_{s-1}+1\leq j \leq k_s,\vspace*{.2cm}\\
			0, \ k_s +1\leq j\leq k_q,
		\end{cases}
	\end{equation*} 
	provided $o(h)/h\to0$ as $h\to0$. Equivalently, if we assume $\lambda_j^{(i)}=0$ for all $j>k_i$. Then, for any $1\leq s\leq q$, we have
	$
	\mathrm{Pr}\{M_s(h)=j| \mathscr{M}^q(h)=j\}={\lambda^{(s)}_{j}}/{\sum_{i=1}^{q}\lambda^{(i)}_j}, \ 1\leq j \leq k_q.
	$
\end{remark}
\begin{remark}
	For $k_1=k_2=\dots=k_q=k$ (say), the jump probabilities of $\{\mathscr{M}^q(t)\}_{t\geq0}$ are given by 
	\begin{equation*}
		\mathrm{Pr}\{\mathscr{M}^q(t,t+h]=i+j|\mathscr{M}^q(t)=i\}=\begin{cases}
			1-\sum_{i=1}^{q}\lambda^{(i)}_jh+o(h),\ j=0,\\
			\sum_{i=1}^{q}\lambda^{(i)}_jh+o(h),\   1\leq j\leq k,\\
			o(h),\ \mathrm{otherwise}.
		\end{cases}
	\end{equation*}
	In this case, for any $1\leq s \leq q$, we have
	$\mathrm{Pr}\{M_s(h)=j| \mathscr{M}^q(h)=j\}=
	{\lambda^{(s)}_{j}}/{\sum_{i=1}^{q}\lambda^{(i)}_j}$, $1\le j\le k$.
\end{remark}
\subsection{The case of merging of countably many GCPs} Let us consider a countable number of independent GCPs $\{M_1(t)\}_{t\geq0}$, $\{M_2(t)\}_{t\geq0}$, $\dots$, such that $\{M_i(t)\}_{t\geq0}$ for each $i\geq1$ performs $k_i$ kinds of jumps with positive rates $\lambda_1^{(i)},\lambda_2^{(i)},\dots,\lambda_{k_i}^{(i)}$, respectively. Without loss of generality, we can take $k_1\leq k_2\leq k_3\leq\dots$. As shown in Section \ref{Section 1.3}, the merged process $\{\mathscr{M}^\infty(t)\}_{t\geq 0}$ is a GCP which performs $k_{(q)}=\max\{k_1,k_2,\dots\}<\infty$ kinds of jumps. Its jump probabilities are given by \eqref{stcde}. For any $s\geq1$, we have
\begin{equation*}
	\mathrm{Pr}\{M_s(h)=j|\mathscr{M}^\infty(h)=j\}=
	\frac{\lambda^{(s)}_j}{\sum_{i=1}^{\infty}\lambda^{(i)}_{j}}, \ 1\leq j \leq k_{(q)},
\end{equation*}
provided $o(h)/h\to0$ as $h\to 0$.

\section{Jumps packet from a particular process}
Here, we consider $q$ independent GCPs  $\{M_1(t)\}_{t\geq0}$, $\{M_2(t)\}_{t\geq0}$, $\dots$, $\{M_q(t)\}_{t\geq0}$, such that $\{M_i(t)\}_{t\geq0}$ for each $1\leq i\leq q$ performs $k_i$ kinds of jumps of amplitude $1,2,\dots,k_i$ with positive rates $\lambda_1^{(i)},\lambda_2^{(i)},\dots,\lambda_{k_i}^{(i)}$, respectively. We assume $\lambda_j^{(i)}=0$ for $ j>k_i$ and for each $1\leq i\leq q$. Let $\{\mathscr{M}^q(t)\}_{t\geq0}$ be their merged process as defined in Section \ref{section 1.2}. For each $i$, let $A_i(t)$ denotes the number of packets of jumps registered by time $t$ in the merged process $\{\mathscr{M}^q(t)\}_{t\geq0}$ from the merging component $\{M_i(t)\}_{t\geq0}$. So, $\{A_i(t),\mathscr{M}^q(t)\}_{t\geq0}$ is a bivariate random process. In an infinitesimal interval of length $h$, the joint transition probabilities of $\{A_i(t),\mathscr{M}^q(t)\}_{t\geq0}$ are given by 
\begin{multline}\label{transcompo}
	\mathrm{Pr}\{A_i(t+h)=a+i,\mathscr{M}^q(t+h)=n+j|A_i(t)=a,\mathscr{M}^q(t)=n\}\\=\begin{cases}
		1-\Lambda h+o(h),\ i= j=0,\vspace{.1cm}\\
		\lambda_j^{(i)}h+o(h),\ i=1,\,1\leq j\leq k_i,\vspace{.1cm}\\
		\sum_{\substack{k=1\\ k\neq i}}^{q}\lambda_j^{(k)}h+o(h),\, i=0,\,1\leq j\leq k_q,\vspace{.1cm} \\
		o(h),\, \text{otherwise},
	\end{cases}
\end{multline} 
where $\Lambda=\sum_{j=1}^{k_q}\sum_{i=1}^{q}\lambda_j^{(i)}.$
\begin{proposition}
	The joint pmf $p(a,n,t)=\mathrm{Pr}\{A_i(t)=a,\mathscr{M}^q(t)=n\}$ solve the following system of differential equations:
	{\small	\begin{align}\label{compde}		\frac{\mathrm{d}}{\mathrm{d}t}p(a,n,t)&=-\Lambda p(a,n,t)+\sum_{j=1}^{k_i}\lambda_j^{(i)}p(a-1,n-j,t) +\sum_{l=1}^{k_q}\sum_{\substack{k=1\\
					k\neq i}}^{q}\lambda_l^{(k)}p(a,n-l,t),\, 0\leq a\leq n,
	\end{align}}
	\normalsize	with initial conditions
	\begin{equation}\label{compini}
		p(a,n,0)=\begin{cases}
			1, \, a=n=0,\\
			0,\, \text{otherwise}.
		\end{cases}
	\end{equation}
\end{proposition}   
\begin{proof}
	We have\small
	\begin{align*}
		p(a,n&,t+h)=p(a,n,t)\mathrm{Pr}\{A_i(t+h)=a,\mathscr{M}^q(t+h)=n|A_i(t)=a,\mathscr{M}^q(t)=n\}\\&\ \ +\sum_{j=1}^{k_i}p(a-1,n-j,t)\mathrm{Pr}\{A_i(t+h)=a,\mathscr{M}^q(t+h)=n|A_i(t)=a-1,\mathscr{M}^q(t)=n-j\}\\&\ \ +\sum_{l=1}^{k_q}p(a,n-l,t)\mathrm{Pr}\{A_i(t+h)=a,\mathscr{M}^q(t+h)=n|A_i(t)=a,\mathscr{M}^q(t)=n-l\}.
	\end{align*}
	\normalsize		
	Using \eqref{transcompo}, we get
	\begin{equation*}
		p(a,n,t+h)=p(a,n,t)\left(1-\Lambda h\right)+\sum_{j=1}^{k_i}p(a-1,n-j,t)\lambda_j^{(i)}h+\sum_{l=1}^{k_q}\sum_{\substack{k=1\\ k\neq i}}^{q}p(a,n-l,t)\lambda_l^{(k)}h+o(h).
	\end{equation*}
	On rearranging the terms and letting  $h\to 0$, we get the required result.
\end{proof}
\begin{proposition}
	The joint pgf $G(u,v,t)=\mathbb{E}\left(u^{A_i(t)}v^{\mathscr{M}^q(t)}\right)$, $ |u|\leq1$, $ |v|\leq1$ solves the following differential equation:
	\begin{equation}\label{pgfcomp}
		\frac{\partial}{\partial t}G(u,v,t)=\Bigg(-\Lambda+\sum_{j=1}^{k_i}\lambda_j^{(i)}uv^j+\sum_{l=1}^{k_q}\sum_{\substack{k=1\\ k\neq i}}^{q}\lambda_l^{(k)}v^l\Bigg)G(u,v,t),
	\end{equation}
	with initial condition $G(u,v,0)=1$.
\end{proposition}
\begin{proof}
	From \eqref{compini}, we have $G(u,v,0)=\sum_{n=0}^{\infty}\sum_{a=0}^{n}u^{a}v^np(a,n,0)=1$.
	On multiplying $u^{a}v^n$ on both sides of \eqref{compde}, and summing over the range of $a= 0,1,\dots,n$ and $n=0,1,\dots$, we get the required result.
\end{proof}	

On solving \eqref{pgfcomp}, we get 
\begin{equation}\label{pgfam}
	G(u,v,t)=\exp\Bigg(-\Lambda +\sum_{j=1}^{k_i}\lambda_j^{(i)}uv^j+\sum_{l=1}^{k_q}\sum_{\substack{k=1\\ k\neq i}}^{q}\lambda_l^{(k)}v^l\Bigg)t. 
\end{equation}

\begin{theorem}
	The joint pmf $p(a,n,t)$ of $\{A_i(t),\mathscr{M}^q(t)\}_{t\geq0}$ is given by
	\begin{equation}\label{jointpmfam}
		p(a,n,t)=\sum_{\Theta(k_i,k_q,a,n)}\prod_{j=1}^{k_i}\prod_{l=1}^{k_q}\frac{(\lambda_j^{(i)}t)^{r_j}}{r_j!s_l!}\Bigg(\sum_{\substack{k=1\\ k\neq i}}^{q}\lambda_l^{(k)}t\Bigg)^{s_l}e^{-\Lambda t},
	\end{equation}
	where \scriptsize
	$
	\Theta(k_i,k_q,a,n)=\left\{(r_1,r_2,\dots,r_{k_i},s_1,s_2,\dots,s_{k_q})\in\mathbb{N}_0^{k_i+k_q}:\sum_{j=1}^{k_i}r_j=a,\,\sum_{j=1}^{k_i}jr_j+\sum_{l=1}^{k_q}ls_l=n\right\}.
	$
\end{theorem}	
\normalsize\begin{proof}
	From \eqref{pgfam}, we have
	\begin{equation*}
		G(u,v,t)=e^{-\Lambda t}\sum_{r=0}^{\infty}\frac{1}{r!}\Bigg(\sum_{j=1}^{k_i}\lambda_j^{(i)}uv^j+\sum_{l=1}^{k_q}\sum_{\substack{k=1\\ k\neq i}}^{q}\lambda_l^{(k)}v^l\Bigg)^rt^r.
	\end{equation*}
	On applying multinomial theorem, we get
	\begin{equation*}
		G(u,v,t)=e^{-\Lambda t}\sum_{r=0}^{\infty}\sum_{S(k_i,k_q,r)}\prod_{j=1}^{k_i}\prod_{l=1}^{k_q}\frac{(\lambda_j^{(i)}t)^{r_j}}{r_j!s_l!}\Bigg(\sum_{\substack{k=1\\ k\neq i}}^{q}\lambda_l^{(k)}t\Bigg)^{s_l}u^{r_j}v^{jr_j+ls_l},
	\end{equation*}
	where $S(k_i,k_q,r)=\{(r_1,\dots,r_{k_i},s_1,\dots,s_{k_q}): \sum_{j=1}^{k_i}r_j+\sum_{l=1}^{k_q}s_l=r,\, r_j,s_l\in\{0,1,\dots,r\}\}$. Finally, on rearranging the terms and collecting the coefficient of $u^{a}v^n$, we get the required result.
\end{proof}	
\begin{remark}\label{Atpmf}
	On substituting $v=1$ in \eqref{pgfam}, we get the marginal pgf $G(u,t)=\mathbb{E}(u^{A_i(t)})$ of $\{A_i(t)\}_{t\geq0}$ in the following form:
	\begin{equation*}
		G(u,t)=e^{-\lambda^{(i)} t(1-u)},\, |u|\leq 1,
	\end{equation*}
	where $\lambda^{(i)}=\lambda^{(i)}_1+\lambda^{(i)}_2+\dots+\lambda^{(i)}_{k_i}$. Thus, for each $t\geq0$, $A_i(t)$ has Poisson distribution with parameter $\lambda^{(i)} t$. This can also be established by summing \eqref{jointpmfam} over the range $n=0,1,\dots$. Also, the process $\{A_i(t)\}_{t\geq0}$ has independent and stationary increments. Thus, $\{A_i(t)\}_{t\geq0}$ for each $1\leq i\leq q$ are independent Poisson processes with parameter $\lambda^{(i)}t$, respectively. So, the process $\{\mathscr{A}^q(t)\}_{t\geq0}$ that denotes the number of packets of jumps registered in the merged process $\{\mathscr{M}^q(t)\}_{t\geq0}$ by time $t$ is a Poisson process with parameter $\sum_{i=1}^{q}\lambda^{(i)}t$.
	Now,
	\begin{align}
		\mathrm{Pr}\{A_i(t)=a|\mathscr{A}^q(t)=b\}&=\frac{\mathrm{Pr}\{A_i(t)=a\}}{\mathrm{Pr}\left\{\sum_{j=1}^{q}A_j(t)=b\right\}}\mathrm{Pr}\Bigg\{\displaystyle\sum_{\substack{j=1\\j\ne i}}^{q}A_j(t)=b-a\Bigg\}\nonumber\\
		&=\frac{b!}{a!(b-a)!}\frac{\left(\lambda^{(i)}\right)^a}{\left(\sum_{j=1}^{q}\lambda^{(j)}\right)^b}\Bigg(\sum_{\substack{j=1\\j\ne i}}^{q}\lambda^{(j)}\Bigg)^{b-a},\label{binopmf}
	\end{align}
	where we have used the independence of $\{A_i(t)\}_{t\geq0}$ in the first step. Thus, conditional on the event $\{\mathscr{A}^q(t)=b\}$, the process $\{A_i(t)\}_{t\geq0}$ follows binomial distribution, that is, $A_i(t)\sim Bin\left(b,\frac{\lambda^{(i)}}{\sum_{j=1}^{q}\lambda^{(j)}}\right)$.

	On taking derivative of \eqref{pgfam} with respect to $u$ and $v$, we get
	\begin{align*}
		\frac{\partial^2}{\partial v\, \partial u}G(u,v,t)&=\left(\left(\sum_{j=1}^{k_i}j\lambda_j^{(i)}uv^{j-1}t+\sum_{l=1}^{k_q}\sum_{k=1}^{q}l\lambda_l^{(k)}v^{l-1}t\right)^2\right.\\
		&\ \ \left. + \left(\sum_{j=1}^{k_i}j(j-1)\lambda_j^{(i)}uv^{j-2}t+\sum_{l=1}^{k_q}\sum_{k=1}^{q}l(l-1)\lambda_l^{(k)}v^{l-2}t\right)\right)G(u,v,t).
	\end{align*}
	Therefore,
	\begin{align*}
		\mathbb{E}\left(A_i(t)\,\mathscr{M}^q(t)\right)&=\frac{\partial^2}{\partial v\,\partial u}G(u,v,t)\bigg|_{u=v=1}
		=\sum_{j=1}^{k_i}j\lambda_j^{(i)}t+\lambda^{(i)}\sum_{l=1}^{k_q}\sum_{k=1}^{q}l\lambda_l^{(k)}t^2.
	\end{align*}
	As the merged process $\{\mathscr{M}^q(t)\}$ is a GCP, thus from \eqref{meanvar}, we have $ \mathbb{E}(\mathscr{M}^q(t))=\sum_{j=1}^{k_q}\sum_{i=1}^{q}$ $j\lambda_j^{(i)}t$, $\operatorname{Var}(\mathscr{M}^q(t))=\sum_{j=1}^{k_q}\sum_{i=1}^{q}j^2\lambda_j^{(i)}t$. Also, from Remark \ref{Atpmf}, $\mathbb{E}(A_i(t))=\operatorname{Var}(A_i(t))=\lambda^{(i)}t$. Hence, the covariance between $\{A_i(t)\}_{t\geq0}$ and $\{\mathscr{M}^q(t)\}_{t\geq0}$ is given by
	\begin{equation*} \operatorname{Cov}\left(A_i(t),\mathscr{M}^q(t)\right)=\mathbb{E}(A_i(t)\,\mathscr{M}^q(t))-\mathbb{E}(A_i(t))\mathbb{E}(\mathscr{M}^q(t))=\sum_{j=1}^{k_i}j\lambda_j^{(i)}t
	\end{equation*}
	and their correlation coefficient is 
	\begin{equation*}
		\operatorname{Corr}\left(A_i(t),\mathscr{M}^q(t)\right)=\frac{\sum_{j=1}^{k_i}j\lambda_j^{(i)}}{\sqrt{\lambda^{(i)}}\sqrt{\sum_{j=1}^{k_q}\sum_{i=1}^{q}j^2\lambda_j^{(i)}}},
	\end{equation*}
	which is a constant and does not vary with time.
\end{remark}

\section{Splitting of jumps in GCP}
In this section, we discuss the splitting of a GCP into two components and then in $q\ge2$ split components. First, we consider the Type I splitting.
\subsection{Splitting of jumps in GCP- Type I}\label{section2}
Here, we study the splitting of jumps in a GCP. The simultaneous jumps in the split components are not allowed. First, we study the splitting in two component processes and then extend it to the case of $q$ component processes, where $q\geq2$ is any positive integer. Let  $\{M(t)\}_{t\geq0}$ be a GCP which performs $k$ kinds of jumps of amplitude $1,2,\dots, k$ with positive rates $\lambda_1,\lambda_2,\dots, \lambda_k$, respectively. 
\subsubsection{Splitting of jumps with a coin flip}
Let $p$-coin be a coin in which head appears with probability $p$. 
If a jump of size $j\in\{1,2,\dots,k\}$ gets register in the GCP $\{M(t)\}_{t\ge0}$ then a $p$-coin is tossed. If head appears then this $j$-size jump gets register in the first component process $\{M_1(t)\}_{t\geq0}$  otherwise it gets register in the second component process  $\{M_2(t)\}_{t\geq0}$.  Assuming that the outcomes of this coin flips are independent of $\{M(t)\}_{t\geq0}$, each $j$-size jump is routed independently to two component processes $\{M_1(t)\}_{t\geq0}$ and $\{M_2(t)\}_{t\geq0}$ with probability $p$ and $1-p$, respectively. 

In an infinitesimal  time interval of length $h$, such that $o(h)/h \to 0$ as $h\to0$, a jump of size $j\in\{1,2,\dots,k\}$ gets register in the process $\{M(t)\}_{t\geq0}$ with probability $\lambda_j h+o(h)$. In such a time interval, a $j$-size jump gets register either in  $\{M_1(t)\}_{t\geq0}$ with probability $\lambda_jph+o(h)$ or in $\{M_2(t)\}_{t\geq0}$    with probability $\lambda_j(1-p)h+o(h)$. This type of splitting is illustrated in Figure \ref{splitfig1}.

\begin{figure}[htp]
	\centering	\includegraphics[width=12cm]{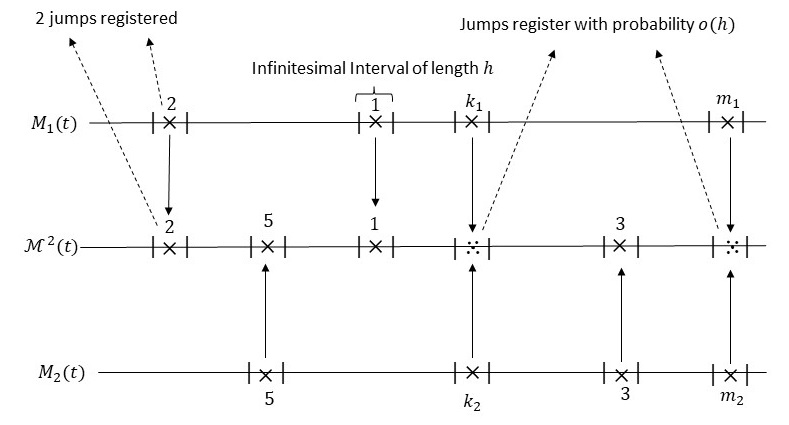}
	\caption{Splitting of GCP with $k=5$.}\label{splitfig1}	\end{figure}
\begin{remark}
	Observe that each $j$-size jump of the process $\{M(t)\}_{t\geq0}$ can be routed independently to either of the split component processes $\{M_1(t)\}_{t\geq0}$ or $\{M_2(t)\}_{t\geq0}$ according to independent Bernoulli trials with parameter $p$. Thus, splitting is governed by a Bernoulli process which is independent of the GCP $\{M(t)\}_{t\geq0}$.
\end{remark}
As the process $\{M(t)\}_{t\geq0}$ has independent and stationary increments and the division of jumps in the split components happens with independent coin flips, so the split component processes $\{M_1(t)\}_{t\geq0}$ and $\{M_2(t)\}_{t\geq0}$ have independent and stationary increments. Next, we explicitly show this in the following result:

Let $X(t_1,t_2]$ denote the increments of a counting process $\{X(t)\}_{t\geq0}$, that is,	$X(t_1,t_2]\coloneqq X(t_2)-X(t_1)$,  $t_2>t_1\geq0$. 
\begin{proposition}\label{prpindpst}
	The split component processes  $\{M_1(t)\}_{t\geq0}$ and $\{M_2(t)\}_{t\geq0}$ have independent and stationary increments.   	
\end{proposition}
\begin{proof}
	Let $0=t_0<t_1<\dots<t_r$. Consider
	$M(t_{l-1},t_l]=n_l=\sum_{j=1}^{k}j(n_l)_j$, $M_1(t_{l-1},t_l]= n_{1l}=\sum_{j=1}^{k}j(n_{1l})_j$ and $M_2(t_{l-1},t_l]= n_{2l}=\sum_{j=1}^{k}j(n_{2l})_j$ with $(n_l)_j=(n_{1l})_j+(n_{2l})_j$. Here, $(n_l)_j$, $(n_{1l})_j$ and $(n_{2l})_j$ are the total number of $j$-size jumps of $\{M(t)\}_{t\geq0}$,  $\{M_1(t)\}_{t\geq0}$ and $\{M_2(t)\}_{t\geq0}$ in $(t_{l-1},t_l]$, respectively. 
	Conditional on the total number of $j$-size jumps of $\{M(t)\}_{t\geq0}$ in interval $(t_{l-1},t_l]$, we have
	{\small
		\begin{align}
			\mathrm{Pr}\{M_1(t_{l-1},t_l]=j(n_{1l})_j,M_2(t_{l-1},t_l]=j(n_{2l})_j\big|M(t_{l-1},t_l]=j(n_{l})_j\}=\binom{(n_{l})_j}{(n_{1l})_j}p^{(n_{1l})_j}(1-p)^{(n_{2l})_j}.\label{conbin}
	\end{align}}
	Equation \eqref{conbin} is simply a binomial distribution because each $j$-size jump in $\{M(t)\}_{t\geq0}$ gets register independently in the process $\{M_1(t)\}_{t\geq0}$ with probability $p$. Thus, on using \eqref{conbin} for each time interval $(t_{l-1},t_l]$, $l\in\{1,2,\dots,r\}$, we have following conditional probability:
	{\small	\begin{align}
			\mathrm{Pr}&\{M_1(t_{l-1},t_{l}]=n_{1l},M_2(t_{l-1},t_l]=n_{2l}\big|M(t_{l-1},t_l]=n_{l}\}\nonumber\\
			&=\sum_{\Theta(k,n_{1l},n_{2l})}\mathrm{Pr}\bigg\{M_1(t_{l-1},t_{l}]=\sum_{j=1}^{k}j(n_{1l})_j,M_2(t_{l-1},t_l]=\sum_{j=1}^{k}j(n_{2l})_j\bigg|M(t_{l-1},t_l]=\sum_{j=1}^{k}j(n_{l})_j\bigg\}\nonumber\\
			& =\sum_{\Theta(k,n_{1l},n_{2l})}\prod_{j=1}^{k}\binom{(n_{l})_j}{(n_{1l})_j}p^{(n_{1l})_j}(1-p)^{(n_{2l})_j},\label{condiprbb}
	\end{align}}
	where
	{\small	\begin{equation*}
			\Theta(k,n_{1l},n_{2l})=\bigg\{(((n_{il})_j)_{j=1}^k)_{i=1}^2:\sum_{j=1}^{k}j(n_{1l})_j=n_{1l},\\ \sum_{j=1}^{k}j(n_{2l})_j=n_{2l}, (n_{1l})_j\geq0,\, (n_{2l})_j\geq0\bigg\}.
	\end{equation*}}
	
	\noindent Now, consider
	{\small\begin{align*}
			&\mathrm{Pr}\{M_1(t_0,t_1]= n_{11},M_1(t_1,t_2]=n_{12},\dots,M_1(t_{r-1},t_r]= n_{1r}\}\\
			&=\sum_{n_{21}\geq 0,\dots,n_{2r}\geq0}\mathrm{Pr}\{M_1(t_0,t_1]= n_{11},\dots,M_1(t_{r-1},t_r]= n_{1r},M_2(t_0,t_1]= n_{21},\dots,M_2(t_{r-1},t_r]= n_{2r}\}\\
			&=\sum_{n_{21}\geq 0,\dots,n_{2r}\geq0}\mathrm{Pr}\left( \cap_{l=1}^{r}\{M_1(t_{l-1},t_{l}]=n_{1l},M_2(t_{l-1},t_l]=n_{2l}\}\bigg|\cap_{l=1}^{r}\{M(t_{l-1},t_l]=n_{l}\}\right)\\
			&\hspace{3cm} \cdot\mathrm{Pr}\left( \cap_{l=1}^{r}\left\{M(t_{l-1},t_l]= n_{l}\right\}\right).
	\end{align*}}
	On using \eqref{condiprbb} and the independent increments property of $\{M(t)\}_{t\geq0}$, the last step reduces to
	\begin{align*}
		\mathrm{Pr}\{&M_1(t_0,t_1]= n_{11},M_1(t_1,t_2]=n_{12},\dots,M_1(t_{r-1},t_r]= n_{1r}\}\\
		&=\sum_{n_{21}\geq 0,\dots,n_{2r}\geq0}\mathrm{Pr}\left( \cap_{l=1}^{r}\{M(t_{l-1},t_l]= n_{l}\}\right)\prod_{l=1}^{r}\sum_{\Theta(k,n_{1l},n_{2l})}\prod_{j=1}^{k}\binom{(n_{l})_j}{(n_{1l})_j}p^{(n_{1l})_j}(1-p)^{(n_{2l})_j}\\
		&=\prod_{l=1}^{r}\sum_{n_{2l}\geq0}\mathrm{Pr}\{M(t_{l-1},t_l]= n_{l}\}\sum_{\Theta(k,n_{1l},n_{2l})}\prod_{j=1}^{k}\binom{(n_{l})_j}{(n_{1l})_j}p^{(n_{1l})_j}(1-p)^{(n_{2l})_j}\\
		&=\prod_{l=1}^{r}\mathrm{Pr}\{M_1(t_{l-1},t_l]= n_{1l}\}.
	\end{align*}
	Thus, the first component process $\{M_1(t)\}_{t\geq0}$ has independent increments.
	
	For some $s>0$, we let
	$M(t,t+s]=n=\sum_{j=1}^{k}j(n)_j$,    $M_1(t,t+s]=n_1=\sum_{j=1}^{k}j(n_1)_j$   and $M_2(t,t+s]=n_2=\sum_{j=1}^{k}j(n_2)_j
	$
	with $(n)_j=(n_1)_j+(n_2)_j$, where $(n)_j$, $(n_1)_j$ and $(n_2)_j$ are the total number of $j$-size jumps of $\{M(t)\}_{t\geq0}$,  $\{M_1(t)\}_{t\geq0}$ and $\{M_2(t)\}_{t\geq0}$ in  $(t,t+s]$, respectively. Using the similar argument, that is, each $j$-size jump of $\{M(t)\}_{t\geq0}$ independently registers in  $\{M_1(t)\}_{t\geq0}$ with probability $p$, we have following conditional probability: 
	{\small	\begin{align}
			\mathrm{Pr}&\{M_1(t,t+s]=n_{1},M_2(t,t+s]=n_{2}\big|M(t,t+s]=n\}\nonumber\\
			&=\sum_{\Theta(k,n_{1},n_{2})}\mathrm{Pr}\bigg\{M_1(t,t+s]=\sum_{j=1}^{k}j(n_{1})_j,M_2(t,t+s]=\sum_{j=1}^{k}j(n_{2})_j\bigg|M(t,t+s]=\sum_{j=1}^{k}j(n)_j\bigg\}\nonumber\\
			& =\sum_{\Theta(k,n_{1},n_{2})}\prod_{j=1}^{k}\binom{(n)_j}{(n_{1})_j}p^{(n_{1})_j}(1-p)^{(n_{2})_j},\label{condiprbbb}
	\end{align}}
	where 
	\begin{equation*} \Theta(k,n_1,n_2)=\bigg\{\big(((n_i)_j)_{j=1}^{k}\big)_{i=1}^2:\sum_{j=1}^{k}j(n_1)_{j}=n_1, \sum_{j=1}^{k}j(n_2)_j=n_2,\, (n_1)_j\geq0, (n_2)_j\geq0 \bigg\}.
	\end{equation*}
	Now, using  \eqref{condiprbbb} and the stationary increments property of $\{M(t)\}_{t\geq0}$, we have
	\begin{align*}
		\mathrm{Pr}&\{M_1(t,t+s]=n_1\}=\sum_{n_2\geq 0}\mathrm{Pr}\{M_1(t,t+s]=n_1,M_2(t,t+s]=n_2\}\\
		&=\sum_{n_2\geq 0}\mathrm{Pr}\{M_1(t,t+s]=n_{1},M_2(t,t+s]=n_{2}\big|M(t,t+s]=n\}\mathrm{Pr}\{M(t,t+s]=n\}\\
		&=\sum_{n_2\geq 0}\mathrm{Pr}\{M(t,t+s]=n\}\sum_{\Theta(k,n_1,n_2)}\prod_{j=1}^{k}\binom{(n)_j}{(n_1)_j}p^{(n_1)_j}(1-p)^{(n_2)_j}\\
		&=\sum_{n_2\geq 0} \mathrm{Pr}\{M(s)=n\}\sum_{\Theta(k,n_1,n_2)}\prod_{j=1}^{k}\binom{(n)_j}{(n_1)_j}p^{(n_1)_j}(1-p)^{(n_2)_j}\\
		&=\mathrm{Pr}\{M_1(s)=n_1\}.
	\end{align*}
	\normalsize	Hence, the  process $\{M_1(t)_{t\geq0}\}$ has stationary increments. Similarly, it can be shown that the results hold true for $\{M_2(t)\}_{t\geq0}$.	 
\end{proof}

Next, we show the independence of split component processes. It is sufficient to show their independence in the interval $(t,t+h]$, where $h$ is infinitesimal time duration. If the split component processes do not have simultaneous jumps in such a time interval then they cannot be independent. However, simultaneous jumps in the split components happen with probability of order $o(h)$. So, without ignoring this probability, we have the following result:
\begin{proposition}\label{prpind}
	The split component processes $\{M_1(t)\}_{t\geq0}$ and $\{M_2(t)\}_{t\geq0}$ are independent.
\end{proposition}
\begin{proof}
	Let $r>0$ and $0=t_0\leq t_1\leq t_2\leq\dots\leq t_r$. It is enough to show that the processes $\{M_1(t_l);\, 1\leq l\leq r \}$ and $\{M_2(t_q);\ 1\leq q\leq r \}$ are independent. Equivalently, it is sufficient to establish that $\{M_{1}(t_{l-1},t_l];\, 1\leq l\leq r\}$ and $\{M_{2}(t_{q-1},t_q];\, 1\leq q\leq r\}$ are independent.
	
	For $l=q$,  let $t_l=t_q=t>0$ and
	$
	M(t)=n=\sum_{j=1}^{k}j(n)_j$,  $M_1(t)=n_1=\sum_{j=1}^{k}j(n_1)_j$ and $M_2(t)=n_2=\sum_{j=1}^{k}j(n_2)_j$ 
	with $(n)_j=(n_1)_j+(n_2)_j$, where $(n)_j,\, (n_1)_j$ and $ (n_2)_j$ are total number of  $j$-size jumps that get register in $\{M(t)\}_{t\geq0}$, $ \{M_1(t)\}_{t\geq0}$ and $\{M_2(t)\}_{t\geq0}$ by time $t$, respectively.
	Since, each $j$-size jump gets register independently in the process  $\{M_1(t)\}_{t\geq0}$ with probability $p$, we have
	\begin{align*}
		\mathrm{Pr}\bigg\{M_1(t)=j(n_1)_j,\,  M_2(t)=j(n_2)_j\bigg|M(t)=j(n)_j\bigg\} 
		&=\binom{(n)_j}{(n_1)_j}p^{(n_1)_j}(1-p)^{(n_2)_j}.
	\end{align*}
	So, 
	\begin{align}
		\mathrm{Pr}\{M_1(t)&=n_1,\,  M_2(t)=n_2\big|M(t)=n\}\nonumber\\
		&=\sum_{\Theta(k,n_1,n_2)}\mathrm{Pr}\bigg\{M_1(t)=\sum_{j=1}^{k}j(n_1)_j,\,  M_2(t)=\sum_{j=1}^{k}j(n_2)_j\bigg|M(t)=\sum_{j=1}^{k}j(n)_j\bigg\}\nonumber\\
		&=\sum_{\Theta(k,n_1,n_2)}\prod_{j=1}^{k}\binom{(n)_j}{(n_1)_j}p^{(n_1)_j}(1-p)^{(n_2)_j},\label{condiprobb}
	\end{align}
	where 	\begin{equation*} \Theta(k,n_1,n_2)=\bigg\{\big(((n_i)_j)_{j=1}^k\big)_{i=1}^2:\sum_{j=1}^{k}j(n_1)_{j}=n_1, \sum_{j=1}^{k}j(n_2)_j=n_2,\, (n_1)_j\geq0, (n_2)_j\geq0 \bigg\}.
	\end{equation*}
	As the event $\{M_1(t)=n_1,\,  M_2(t)=n_2\}$ is contained in $\{M(t)=n\}$. Thus, on using \eqref{gcpnt} and \eqref{condiprobb}, we have
	\begin{align*}
		\mathrm{Pr}	\{M_1(t)&=n_1, M_2(t)=n_2\}=	\mathrm{Pr}\{M_1(t)=n_1,\,  M_2(t)=n_2\big|M(t)=n\}\mathrm{Pr}\{M(t)=n\}\\
		&=\sum_{\Theta(k,n_1,n_2)}\mathrm{Pr}\bigg\{M_1(t)=\sum_{j=1}^{k}j(n_1)_j,\,  M_2(t)=\sum_{j=1}^{k}j(n_2)_j\bigg|M(t)=\sum_{j=1}^{k}j(n)_j\bigg\}\\
		&\ \ \ \ \ \  \cdot\mathrm{Pr}\bigg\{M(t)=\sum_{j=1}^{k}j(n)_j\bigg\}\\
		&=\sum_{\Theta(k,n_1,n_2)}\prod_{j=1}^{k}\binom{(n)_j}{(n_1)_j}p^{(n_1)_j}(1-p)^{(n_2)_j}\prod_{j=1}^{k}\frac{(\lambda_jt)^{(n)_j}}{(n)_j!}e^{-\lambda_j t}\\
		&=\sum_{\Theta(k,n_1,n_2)}\prod_{j=1}^{k}\binom{(n)_j}{(n_1)_j}p^{(n_1)_j}(1-p)^{(n_2)_j}\frac{(\lambda_jt)^{(n)_j}}{(n)_j!}e^{-\lambda_j t}\\
		&=\sum_{\Omega(k,n_1)}\prod_{j=1}^{k}\frac{(\lambda_j pt)^{(n_1)_j}}{(n_1)_j!}e^{-\lambda_j pt}\sum_{\Omega(k,n_2)}\prod_{j=1}^{k}\frac{(\lambda_j(1-p)t)^{(n_2)_j}}{(n_2)_j!}e^{-\lambda_j(1-p)t}\\
		&=\mathrm{Pr}\{M_1(t)=n_1\}\mathrm{Pr}\{M_2(t)=n_2\},
	\end{align*}
	where  $\Omega(k,n_i)=\{((n_i)_1,(n_i)_2,\dots,(n_i)_k):\sum_{j=1}^{k}j(n_i)_j=n_i,\, (n_i)_j\geq0 \}$, $i=1,2$. 
	For $l\neq q$, the independence of split components follows from the independent increments of the GCP $\{M(t)\}_{t\geq0}$. This completes the proof.
\end{proof}
\begin{remark}
	We conclude that the split component processes $\{M_1(t)\}_{t\geq0}$ and $\{M_2(t)\}_{t\geq0}$ are independent GCPs which perform $k$ kinds of jumps of size $j\in\{1,2,\dots,k\}$ with positive rates $\lambda_j p$ and $\lambda_j (1-p)$, respectively. 
\end{remark}
For $k=1$, the above result reduces to the case of Poisson process, that is, the split components are Poisson processes.
\subsubsection{Splitting of jumps on rolling a $q$-faced die}
In this case, we split the $j$-size jumps into $q$ component processes. Whenever a jump of size $j$ gets register in the process $\{M(t)\}_{t\ge0}$, we roll a $q-$faced die. If $i$ appears then the $j$-size jump is considered to get register in the process $\{M_i(t)\}_{t\geq0}$. Let $p_i$ be the probability that $i$ appears on this $q$-faced die. So, $\sum_{i=1}^{q}p_i=1$. It is assumed that the outcomes on rolling this $q$-faced die are independent of $\{M(t)\}_{t\geq0}$. Thus, in an infinitesimal time interval of length $h$, a jump of size $j\in\{1,2,\dots,k\}$ registers in the process $\{M(t)\}_{t\geq0}$ with probability $\lambda_jh+o(h)$ which then gets register in the component process $\{M_i(t)\}_{t\geq0}$, $i\in\{1,2,\dots,q\}$ with probability  $\lambda_jp_ih+o(h)$. 

As the process $\{M(t)\}_{t\geq0}$ has independent and stationary increments and the splitting of jumps into $q$ components happens with the independent rolls of a $q$-faced die, so for each $i\in\{1,2,\dots,q\}$ the split component process $\{M_i(t)\}_{t\geq0}$ has independent and stationary increments. 
\begin{proposition}
	The split component processes  $\{M_i(t)\}_{t\geq0}$, $i\in\{1,2,\dots,q\}$ has independent and stationary increments.    
\end{proposition}
\begin{proof}
	For any $r\in \mathbb{N}$ and $0=t_0<t_1<\dots<t_r$. Let
	$
	M(t_{l-1},t_l]=n_l=\sum_{j=1}^{k}j(n_l)_j$ and $ M_i(t_{l-1},t_l]=n_{il}=\sum_{j=1}^{k}j(n_{il})_j$ with $(n_l)_j=\sum_{i=1}^{q}(n_{il})_j$, where $(n_l)_j$ and $(n_{il})_j$ are the total number of $j$-size jumps of $\{M(t)\}_{t\geq0}$ and $\{M_i(t)\}_{t\geq0}$ in $(t_{l-1},t_l]$, respectively.
	As each $j$-size jump of $\{M(t)\}_{t\geq0}$ gets register independently in the component process $\{M_i(t)\}_{t\geq0}$ with probability $p_i$. So, conditional on the total number of $j$-size jumps of $\{M(t)\}_{t\geq0}$ in interval $(t_{l-1},t_l]$, we have
	\begin{align}
		\mathrm{Pr}\bigg\{\cap_{i=1}^{q}\{M_i(t_{l-1},t_l]=j(n_{il})_j\}\big|M(t_{l-1},t_l]=j(n_{l})_j\bigg\}
		=\prod_{i=1}^{q}\frac{(n_l)_j!}{(n_{il})_j!}p_{i}^{(n_{il})_j}.\label{conbinq}
	\end{align}
	On using equation \eqref{conbinq}, we have the following conditional probability:
	\begin{align}
		\mathrm{Pr}&\bigg(\cap_{i=1}^{q}\{M_i(t_{l-1},t_{l}]=n_{il}\}\big|M(t_{l-1},t_l]=n_l\bigg)\nonumber\\
		&=\sum_{\Theta(k,(n_{il})_{i=1}^{q})}\mathrm{Pr}\bigg(\cap_{i=1}^{q}\big\{M_i(t_{l-1},t_{l}]=\sum_{j=1}^{k}j(n_{il})_j\big\}\bigg|M(t_{l-1},t_l]=\sum_{j=1}^{k}j(n_l)_j\bigg)\nonumber\\
		&=\sum_{\Theta(k,(n_{il})_{i=1}^{q})}\prod_{j=1}^{k}\prod_{i=1}^{q}\frac{(n_l)_j!}{(n_{il})_j!}p_{i}^{(n_{il})_j},\label{condiprbbbb}
	\end{align}
	where
	\begin{equation*}
		\Theta\big(k,(n_{il})_{i=1}^{q}\big)=\bigg\{(((n_{il})_j)_{j=1}^{k})_{i=1}^{q}:\sum_{j=1}^{k}j(n_{il})_j=n_{il},\ \forall\ i=\{1,2,\dots,q\},\ (n_{il})_j\geq 0\bigg\}.
	\end{equation*}
	
	Thus, on following the similar steps as in the proof of Proposition $\ref{prpindpst}$ and  using  \eqref{condiprbbbb} along with the independent increments of $\{M(t)\}_{t\geq0}$, we have 
	\begin{align*}
		\mathrm{Pr}\{M_i(t_0,t_1]&= n_{i1},M_i(t_1,t_2]= n_{i2},\dots,M_i(t_{r-1},t_r]= n_{ir}\}\\
		&=\sum_{\Theta}\mathrm{Pr}\left( \displaystyle \cap_{l=1}^{r}\{M(t_{l-1},t_l]= n_{l}\}\right)\prod_{l=1}^{r}\sum_{\Theta\big(k,(n_{il})_{i=1}^{q}\big)}\prod_{j=1}^{k}\prod_{i=1}^{q}\frac{(n_l)_j!}{(n_{il})_j!}p_i^{(n_{il})_j}\label{eq1}\\
		&=\prod_{l=1}^{r}\sum_{\Theta(l)}\mathrm{Pr}\{M(t_{l-1},t_l]= n_l\}\sum_{\Theta\big(k,(n_{il})_{i=1}^{q}\big)}\prod_{j=1}^{k}\prod_{i=1}^{q}\frac{(n_l)_j!}{(n_{il})_j!}p_i^{(n_{il})_j}\\
		&=\prod_{l=1}^{r}\mathrm{Pr}\{M_i(t_{l-1},t_l]= n_{il}\},
	\end{align*}
	where
	$\Theta=\{n_{ab}\geq0:a\in\{1,2,\dots,i-1,i+1,\dots,q\},\, b\in\{1,2,\dots,r\}\}$ and  $\Theta(l)=\{n_{al}\geq0:a\in\{1,2,\dots,i-1,i+1,\dots,q\}\}$. Thus for each $i$, $\{M_i(t)\}_{t\geq0}$ has independent increments. 
	
	Now, for $s>0$, we have
	$
	M(t,t+s]=n=\sum_{j=1}^{k}j(n)_j\   \text{and} \  M_i(t,t+s]=n_i=\sum_{j=1}^{k}j(n_i)_j
	$
	with $(n)_j=\sum_{i=1}^{q}(n_i)_j$, where $(n)_j$ and $(n_i)_j$ denote the number of $j$-size jumps of $\{M(t)\}_{t\geq0}$ and $\{M_i(t)\}_{t\geq0}$ in $(t,t+s]$, respectively. Similar to \eqref{condiprbbbb}, for time interval $(t,t+s]$, we have the following conditional probability:
	\begin{equation}
		\mathrm{Pr}\bigg(\cap_{i=1}^{q}\{M_i(t,t+s]=n_{i}\}\big|M(t,t+s]=n\bigg)=\sum_{\Theta\big(k,(n_i)_{i=1}^{q}\big)}\prod_{i=1}^{q}\prod_{j=1}^{k}\frac{(n)_j!}{(n_i)_j!}p_i^{(n_i)_j},\label{condist}
	\end{equation}
	where $\Theta\big(k,(n_i)_{i=1}^{q}\big)=\{(((n_i)_j)_{j=1}^{k})_{i=1}^{q}:\sum_{j=1}^{k}j(n_i)_j=n_i,\ \forall i=1,2,\dots,q,\, (n_i)_j\geq0\}$.
	Thus, on using \eqref{condist} and following the similar lines to that of the proof of  Proposition $\ref{prpindpst}$, we have 
	\begin{align*}
		\mathrm{Pr}\{M_i(t,t+s]=n_{i}\}&=\sum_{\Theta'}\mathrm{Pr}\{M(t,t+s]=n\}\sum_{\Theta\big(k,(n_i)_{i=1}^{q}\big)}\prod_{i=1}^{q}\prod_{j=1}^{k}\frac{(n)_j!}{(n_i)_j!}p_i^{(n_i)_j}\\
		&=\sum_{\Theta'}\mathrm{Pr}\{M(s)=n\}\sum_{\Theta\big(k,(n_i)_{i=1}^{q}\big)}\prod_{i=1}^{q}\prod_{j=1}^{k}\frac{(n)_j!}{(n_i)_j!}p_i^{(n_i)_j}\\
		&=\mathrm{Pr}\{M_i(s)=n_i\},
	\end{align*}
	where the penultimate step follows on using the stationary increments property of the process $\{M(t)\}_{t\geq0}$. Here, $\Theta'=\{n_a>0:a\in\{1,2,\dots,i-1,i+1,\dots,q\}\}$. Hence, for each $i$, the split component process $\{M_i(t)\}_{t\geq0}$ has stationary increments.
\end{proof}
\begin{proposition}
	The split component processes $\{M_i(t)\}_{t\geq0},\, 1\leq i\leq q$ are independent.
\end{proposition}
\begin{proof}
	Let $r>0$ and $0=t_0\leq t_1\leq t_2\leq\dots\leq t_r$. It is enough to show that the processes $\{M_i(t_{l_i}):\, 1\leq l_i\leq r \}$, $i=\{1,2,\dots,q\}$ are independent. Equivalently, it is sufficient to establish that the processes $\{M_{i}(t_{l_i-1},t_{l_i}]:\, 1\leq l_i\leq r\}$ are independent. Let  $t_{l_i}=t$ for all $l_i$ and  
	$
	M(t)=n=\sum_{j=1}^{k}j(n)_j,\ M_i(t)=n_i=\sum_{j=1}^{k}j(n_i)_j,
	$
	such that $(n)_j=(n_1)_j+(n_2)_j+\dots+(n_q)_j$. Here, $(n)_j$ and $ (n_i)_j$ are the total number of $j$-size jumps of  $\{M(t)\}_{t\geq0}$ and $ \{M_i(t)\}_{t\geq0}$ by time $t$, respectively.
	Since, different jumps of size $j$ in $\{M(t)\}_{t\geq0}$ independently goes to $\{M_i(t)\}_{t\geq0}$ with probability $p_i$, therefore, conditioning on the number of $j$-size jumps in $\{M(t)\}_{t\geq0}$, we have
	\begin{equation*}
		\mathrm{Pr}\{M_1(t)=j(n_1)_j,\,  M_2(t)=j(n_2)_j,\dots,M_q(t)=j(n_q)_j\big|M(t)=j(n)_j\} = \prod_{i=1}^{q}\frac{(n)_j!}{(n_i)_j!}p_i^{(n_i)_j}.
	\end{equation*}
	So, 
	\begin{align}
		\mathrm{Pr}\bigg(\cap_{i=1}^{q}\big\{M_i(t)=n_i\big\}&\big|M(t)=n\bigg)\nonumber\\
		&=\sum_{\Theta(k,(n_i)_{i=1}^{q})}\mathrm{Pr}\bigg(\cap_{i=1}^{q}\big\{M_i(t)=\sum_{j=1}^{k}j(n_i)_j\big\}\bigg|M(t)=\sum_{j=1}^{k}j(n)_j\bigg)\nonumber\\
		&=\sum_{\Theta(k,(n_i)_{i=1}^{q})}\prod_{j=1}^{k}\prod_{i=1}^{q}\frac{(n)_j!}{(n_i)_j!}p_i^{(n_i)_j},\label{qjprb}
	\end{align}
	where	$\Theta\big(k,(n_i)_{i=1}^{q}\big)=\{(((n_i)_j)_{j=1}^{k})_{i=1}^{q}:\sum_{j=1}^{k}j(n_i)_j=n_i\ \forall i=1,2,\dots,q,\, (n_i)_j\geq0\}$.
	
	As the event $\{M_1(t)=n_1,\,  M_2(t)=n_2,\dots,M_q(t)=n_q\}$ is contained in $\{M(t)=n\}$, we have
	\begin{align*}
		\mathrm{Pr}&\{M_1(t)=n_1, M_2(t)=n_2,\dots,M_q(t)=n_q\}\\
		&=\mathrm{Pr}\bigg(\cap_{i=1}^{q}\big\{M_i(t)=n_i\big\}\big|M(t)=n\bigg)\mathrm{Pr}\{M(t)=n\}\\
		&=\sum_{\Theta(k,(n_i)_{i=1}^{q})}\mathrm{Pr}\bigg(\cap_{i=1}^{q}\big\{M_i(t)=\sum_{j=1}^{k}j(n_i)_j\big\}\bigg|M(t)=\sum_{j=1}^{k}j(n)_j\bigg)\mathrm{Pr}\bigg\{M(t)=\sum_{j=1}^{k}j(n)_j\bigg\}\\
		&=\sum_{\Theta(k,(n_i)_{i=1}^{q})}\prod_{j=1}^{k}\prod_{i=1}^{q}\frac{(n)_j!}{(n_i)_j!}p_i^{(n_i)_j}\prod_{j=1}^{k}\frac{(\lambda_j t)^{(n)_j}}{(n)_j!}e^{-\lambda_jt},\ \text{(using \eqref{qjprb})}\\
		&=\prod_{i=1}^{q}\sum_{\Omega(k,n_i)}\prod_{j=1}^{k}\frac{(\lambda_jp_it)^{(n_i)_j}}{(n_i)_j!}e^{-\lambda_jp_it}.
	\end{align*}
	Here, $\Omega(k,n_i)=\{((n_i)_1,(n_i)_2,\dots,(n_i)_k):\sum_{j=1}^{k}j(n_i)_j=n_i, (n_i)_j\geq0\}$. For different $l_i$, that is, for different time intervals, the independence of split components follows from the independent increments of GCP $\{M(t)\}_{t\geq0}$. Thus, for each $i$, the split component processes $\{M_i(t)\}_{t\geq0}$ are independent.
\end{proof}
\begin{remark}
	We conclude that the split component processes $\{M_i(t)\}_{t\geq0}$, $i\in\{1,2,\dots,q\}$ are independent GCPs which perform $k$ kinds of jumps of amplitude $j\in\{1,2,\dots,k\}$ with positive rates $\lambda_jp_i$. 
\end{remark} 

\subsection{Splitting of jumps in GCP- Type II}\label{sect3}
Here, we consider a different kind of splitting of jumps in a GCP in which there can be simultaneous jumps in the split components. First, we study the splitting in two component processes and then extend it to the case of $q\geq2$ component processes. 
\subsubsection{Splitting of jumps in two component processes} In this case, we split the jumps of GCP $\{M(t)\}_{t\geq0}$ into two component processes with flips of a $p$-coin. If a jump of size $j\in\{1,2,\dots, k\}$ gets register in $\{M(t)\}_{t\geq0}$ then we flip a $p$-coin $j$ many times and choose to divide the jump into two component processes $\{M_1(t)\}_{t\geq0}$ and $\{M_2(t)\}_{t\geq0}$ according to the number of heads and tails that appear on $j$ flips. Thus, the number of this coin flip is random. As assumed in the previous section, the outcomes on the flips of $p$-coin are independent of $\{M(t)\}_{t\geq0}$. If $j_1$ is the total number of heads and $j_2$ is the total number of tails that appear on $j$ flips of $p$-coin then a jump of size $j_1$ gets register in the first  component process $\{M_1(t)\}_{t\geq0}$ and a jump of size $j_2$ gets register in the second component  $\{M_2(t)\}_{t\geq0}$. Thus, the split component processes can perform $k$ kinds of jumps of amplitude $1,2,\dots,k$.  As the GCP $\{M(t)\}_{t\geq0}$ has independent and stationary increments and the division of jumps into two split component processes happens with independent coin flips, the split component processes $\{M_1(t)\}_{t\geq0}$ and $\{M_2(t)\}_{t\geq0}$ have independent and stationary increments. 

In an infinitesimal time-interval of length $h$, the jump probabilities of $\{M_1(t)\}_{t\geq0}$ are given by
\begin{align*}
	\mathrm{Pr}\{M_1(h)=1\}&=\sum_{j=1}^{k}\mathrm{Pr}\{M(h)=j\}\mathrm{Pr}\{M_1(h)=1|M(h)=j\}\\
	&=\sum_{j=1}^{k}\binom{j}{1}p (1-p)^{j-1}\lambda_jh+o(h).\\
	\mathrm{Pr}\{M_1(h)=2\}&=\sum_{j=2}^{k}\mathrm{Pr}\{M(h)=j\}\mathrm{Pr}\{M_1(h)=2|M(h)=j\}\\
	&	=\sum_{j=2}^{k}\binom{j}{2}p^2 (1-p)^{j-2}\lambda_jh+o(h).\\
	&\hspace*{.22cm}\vdots\\
	\mathrm{Pr}\{M_1(h)=k\}&=p^k\lambda_kh+o(h).
\end{align*} 
So, we have the following transition probabilities for the first component process:
\begin{equation}\label{transprb}
	\mathrm{Pr}\{M_1(t+h)=n_1+j_1|M_1(t)=n_1\}=\begin{cases}
		1-\sum_{j=1}^{k}\lambda_j (1-(1-p)^j)h+o(h),\, j_1=0,\vspace{.1cm}\\
		\sum_{j=j_1}^{k}\binom{j}{j_1}p^{j_1}(1-p)^{j-j_1}\lambda_jh +o(h),\, 1\leq j_1\leq k,\vspace{.1cm}\\
		o(h),\, \text{otherwise}.
	\end{cases}
\end{equation}	
Next, we obtain the system of differential equations that governs the state probabilities of the component process $\{M_1(t)\}_{t\geq0}$.
\begin{proposition}\label{prp1}
	The state probabilities $p(n,t)=\mathrm{Pr}\{M_1(t)=n\}$ satisfy the following system of differential equations:
	\begin{equation*}
		\frac{\mathrm{d}}{\mathrm{d}t}p(n,t)=\bigg(-\sum_{j=1}^{k}\lambda_j(1-(1-p)^j)\bigg)p(n,t)+\sum_{j=1}^{k}\sum_{j_1=1}^{j}\lambda_j\binom{j}{j_1}p^{j_1}(1-p)^{j-j_1}p(n-j_1,t),\, n\geq0,
	\end{equation*} 
	with initial conditions $p(0,0)=1$ and $p(n,0)=0$ for all $n \neq 0$.
\end{proposition}	
\begin{proof}
	From independent and stationary increments of $\{M_1(t)\}_{t\geq0}$, we have
	\begin{equation*}
		p(n,t+h)=\sum_{j_1=0}^{k}p(n-j_1,t)p(j_1,h)+o(h).
	\end{equation*}
	On using \eqref{transprb}, we get
	{\small	\begin{equation*}
			p(n,t+h)=p(n,t)\left(1-\sum_{j=1}^{k}\lambda_j (1-(1-p)^j)h\right)+\sum_{j_1=1}^{k}\sum_{j=j_1}^{k}\binom{j}{j_1}p^{j_1}(1-p)^{j-j_1}p(n-j_1,t)\lambda_jh+o(h),
	\end{equation*}}
	which on rearranging the terms and letting $h\to0$ reduces to the required result.
\end{proof}
Thus, we observe that $\{M_1(t)\}_{t\geq0}$ is a GCP which performs $k$ kinds of jumps of size $j_1\in\{1,2,\dots,k\}$ with reduced rates $\sum_{j=j_1}^{k}\binom{j}{j_1}p^{j_1}(1-p)^{j-j_1}\lambda_j$.

Similarly, it can be shown that $\{M_2(t)\}_{t\geq0}$ is a GCP which performs $k$ kinds of jumps of size $j_2\in\{1,2,\dots, k\}$ with positive rates $\sum_{j=j_2}^{k}\binom{j}{j_2}(1-p)^{j_2}p^{j-j_2}\lambda_j$.

Let us consider a bivariate random process $\{(M_1(t), M_2(t))\}_{t\geq0}$ of the split component processes. In an infinitesimal time interval of length $h$, the joint transition probabilities of $\{M_1(t)\}_{t\geq0}$ and  $\{M_2(t)\}_{t\geq0}$ are 
\begin{multline}\label{jointransprb}
	\mathrm{Pr}\{M_1(t+h)=n_1+j_1,M_2(t+h)=n_2+j_2|M_1(t)=n_1,M_2(t)=n_2\}\\=\begin{cases}
		1-\sum_{j=1}^{k}\lambda_jh+o(h),\, j_1=j_2=0,\vspace{.1cm}\\
		\vspace{.1cm}
		\binom{j}{j_1}p^{j_1}(1-p)^{j_2}\lambda_{j}h+o(h),\, 1\leq j=j_1+j_2\leq k,\vspace{.1cm}\\
		o(h),\, \text{otherwise}.
	\end{cases}
\end{multline} 
\begin{proposition}
	The joint pmf $p(n_1,n_2,t)=\mathrm{Pr}\{M_1(t)=n_1,M_2(t)=n_2\},\, n_1\geq0,\,n_2\geq0$ solves the following differential equation:
	\begin{equation}\label{jointdeq}
		\frac{\mathrm{d}}{\mathrm{d}t}p(n_1,n_2,t)=-\sum_{j=1}^{k}\lambda_jp(n_1,n_2,t)+\sum_{j=1}^{k}\binom{j}{j_1}p^{j_1}(1-p)^{j_2}\lambda_{j}p(n_1-j_1,n_2-j_2,t),
	\end{equation}
	with initial conditions
	\begin{equation}\label{joindecond}
		p(n_1,n_2,0)=\begin{cases}
			1,\, n_1+n_2=0,\\
			0,\, n_1+n_2\neq0.
		\end{cases}
	\end{equation}
\end{proposition}
\begin{proof}On using \eqref{jointransprb} and following the similar lines of the proof of Proposition $\ref{prp1}$, we get the required result.
\end{proof}
\begin{proposition}
	The joint pgf $G(u,v,t)=\mathbb{E}\big(u^{M_1(t)}v^{M_2(t)}\big)$, $|u|\leq1,\,|v|\leq1$ solves 
	\begin{equation}\label{joipgf}
		\frac{	\partial}{\partial t}G(u,v,t)=-G(u,v,t)\sum_{j=1}^{k}\lambda_j\left(1-\binom{j}{j_1}p^{j_1}(1-p)^{j_2}u^{j_1}v^{j_2}\right),
	\end{equation}
	with $G(u,v,0)=1$.
\end{proposition}
\begin{proof}
	On using \eqref{joindecond}, we get $G(u,v,0)=\sum_{n_1=0}^{\infty} \sum_{n_2=0}^{\infty}u^{n_1}v^{n_2}p(n_1,n_2,0)=1$. On multiplying \eqref{jointdeq} with $u^{n_1}v^{n_2}$ and summing over the range of $n_1\geq0$ and $n_2\geq 0$, we get the required result.
\end{proof}
On solving \eqref{joipgf}, we get
\begin{equation*}\label{jopgf}
	G(u,v,t)=\exp\left(-\sum_{j=1}^{k}\lambda_jt+\sum_{j=1}^{k}\binom{j}{j_1}(up)^{j_1}(v(1-p))^{j_2}\lambda_{j}t\right).
\end{equation*}
So, 
\begin{align*}
	\mathbb{E}(M_1(t))&=\frac{\partial}{\partial u}G(u,v,t)\Big|_{u=v=1}=\sum_{j=1}^{k}\binom{j}{j_1}p^{j_1}(1-p)^{j_2}j_1\lambda_{j}t,\\
	\mathbb{E}(M_2(t))&=\frac{\partial}{\partial v}G(u,v,t)\Big|_{u=v=1}=\sum_{j=1}^{k}\binom{j}{j_1}p^{j_1}(1-p)^{j_2}j_2\lambda_{j}t,\\
	\mathbb{E}(M_1(t)M_2(t))&=\frac{\partial^2}{\partial u\partial v}G(u,v,t)\Big|_{u=v=1}\\
	&=\sum_{j=1}^{k}\binom{j}{j_1}p^{j_1}(1-p)^{j_2}j_1\lambda_{j}t\bigg(j_2+\sum_{j=1}^{k}\binom{j}{j_1}p^{j_1}(1-p)^{j_2}j_2\lambda_{j}t\bigg).
\end{align*}
Thus, the covariance of $M_1(t)$ and $M_2(t)$ is  
\begin{equation*}
	\operatorname{Cov}(M_1(t),M_2(t))=\mathbb{E}(M_1(t)M_2(t))-\mathbb{E}(M_1(t))\mathbb{E}(M_2(t))=\sum_{j=1}^{k}\binom{j}{j_1}p^{j_1}(1-p)^{j_2}j_1j_2\lambda_jt,
\end{equation*}
which is zero only if $j_1=0$
or $j_2=0$. Therefore, in this type of splitting the split components are not independent. They are independent provided there is splitting of jumps as discussed in Section $\ref{section2}$.

\subsubsection{Splitting jumps on rolling a $q$-faced die}  In this case, we split the jumps of GCP $\{M(t)\}_{t\geq0}$ into $q\geq2$ components with the rolls of a $q$-faced die. Whenever a jump of size $j\in\{1,2,\dots,k\}$ gets register in the process $\{M(t)\}_{t\geq0}$, we roll a $q$-faced die $j$ many times. If $i\in\{1,2,\dots,q\}$ appears $j_i$ many times in $j$ rolls of the die, such that $\sum_{i=1}^{q}j_i=j$,  then a $j_i$-size jump gets  register in the split component process $\{M_i(t)\}_{t\geq0}$. Further, similar to the above discussed cases, we assume that the outcomes on $j$ rolls of this $q$-faced die are independent of $\{M(t)\}_{t\geq0}$. Let $p_i$ be the probability that $i$ appears on the die. Then, in an infinitesimal time-interval of length $h$, such that, $o(h)/h\to0$ as $h\to0$, the transition probabilities of $\{M_i(t)\}_{t\geq0}$ are given by
\begin{equation*}
	\mathrm{Pr}\{M_i(t+h)=n_i+j_i|M_i(t)=n_i\}=\begin{cases}
		1-\sum_{j=1}^{k}\lambda_j\bigg(1-\sum_{\Theta_i}\prod_{\substack{r=1\\r\neq i}}^{q}\frac{j!}{j_r!}p_r^{j_r}\bigg)h+o(h),\, j_i=0,\vspace{.1cm}\\
		\sum_{j=j_i}^{k}\sum_{\Theta'}\prod_{r=1}^{q}\frac{j!}{j_r!}p_r^{j_r}\lambda_jh+o(h), \, 1\leq j_i\leq k,\vspace{.1cm}\\
		o(h),\, \text{otherwise},
	\end{cases}
\end{equation*}
where $\Theta_i=\bigg\{(j_1,j_2,\dots,j_{i-1},j_{i+1},\dots,j_q):\sum_{\substack{r=1,\\r\neq i}}^{q}j_r=j,\, j_r\geq 0\bigg\}$ and $\Theta' =\{(j_1,j_2,\dots,j_q):\sum_{r=1}^{q}j_r=j,\, j_r\geq0\}$. 

As the process $\{M(t)\}_{t\geq0}$ has independent and stationary increments and the splitting of jumps into $q$ components is also independent, therefore, $\{M_i(t)\}_{t\geq0}$ for each $i\in\{1,2,\dots,q\}$ has independent and stationary increments. Thus, $\{M_i(t)_{t\geq0}\}$ is a GCP which performs $k$ kinds of jump of size $j_i\in\{1,2,\dots,k\}$ with positive rates $\sum_{j=j_i}^{k}\sum_{\Theta'}\prod_{r=1}^{q}\frac{j!}{j_r!}p_r^{j_r}\lambda_j$.

\begin{remark}
	In an infinitesimal time interval of length $h$, the $q$-variate random process $\{(M_1(t)$, $M_2(t))$, $\dots$, $M_q(t)\}_{t\geq0}$ of split components has following transition probabilities:
	\begin{align}\label{transqjoin}
		\mathrm{Pr}\bigg(\cap_{i=1}^{q}    \{M_i(t+h)&=n_i+j_i\} \big|\cap_{i=1}^{q}\{M_i(t)=n_i\}\bigg)\nonumber\\
		&=\begin{cases}
			1-\sum_{j=1}^{k}\lambda_jh+o(h),\, \sum_{i=1}^{q}j_i=0,\vspace{.1cm}\\
			\sum_{\Theta'}\prod_{r=1}^{q}\frac{j!}{j_r!}p_r^{j_r}\lambda_jh+o(h), \, 1\leq j=\sum_{i=1}^{q}j_i\leq k,\vspace{.1cm}\\
			o(h),\, \text{otherwise},
		\end{cases}
	\end{align}  
	where $\Theta' =\{(j_1,j_2,\dots,j_q):\sum_{r=1}^{q}j_r=j,\, j_r\geq0\}$. 
\end{remark}
Let  
$\bar{n}=(n_1,n_2,\dots,n_q)$,  $\bar{u}=(u_1,u_2,\dots,u_q)$ and $\bar{j}=(j_1,j_2,\dots,j_q)$ be $q$-tuple row vectors. Let $\bar{0}=(0,0,\dots,0)$  and $\bar{1}=(1,1,\dots,1)$ be $q$-tuple row vectors. Further, $\bar{n}\neq \bar{0}$ implies that $n_i\neq 0$ for atleast one $1\leq i\leq q$. 
\begin{proposition}
	The joint pmf $p(\bar{n},t)=\mathrm{Pr}\{M_1(t)=n_1,M_2(t)=n_2,\dots,M_q(t)=n_q\}$, $n_i\geq0$ for all $i=\{1,2,\dots,q\}$ satisfies the following differential equation:
	\begin{equation*}
		\frac{\mathrm{d}}{\mathrm{d}t}p(\bar{n},t)=-\sum_{j=1}^{k}\lambda_jp(\bar{n},t)+\sum_{j=1}^{k}\sum_{\Theta'}\prod_{r=1}^{q}\frac{j!}{j_r!}p_r^{j_r}\lambda_jp(\bar{n}-\bar{j},t),
	\end{equation*}
	with initial conditions 
	\begin{equation*}
		p(\bar{n},0)=\begin{cases}
			1,\ \bar{n}=\bar{0},\\
			0,\ \bar{n}\neq \bar{0}.
		\end{cases}
	\end{equation*}
\end{proposition}
\begin{proof}
	On using \eqref{transqjoin} and following along the similar lines of the proof of Proposition \ref{prp1}, we obtain the required result.
\end{proof}
Its joint pgf  $G(\bar{u},t)=\mathbb{E}\bigg(\prod_{i=1}^{q}u_i^{M_i(t)}\bigg)$, $|u_i|\leq 1$, $i=\{1,2,\dots,q\}$ of the vector $(M_1(t)$, $M_2(t)$, $\dots$, $M_q(t))$ solves 
\begin{equation}\label{joinqpgff}
	\frac{\partial}{\partial t}G(\bar{u},t)=-G(\bar{u},t)\sum_{j=1}^{k}\lambda_j\bigg(1-\sum_{\Theta'}\prod_{r=1}^{q}\frac{j!}{j_r!}p_r^{j_r}u_r^{j_r}\bigg),
\end{equation}
with initial condition $G(\bar{u},0)=1$.

On solving \eqref{joinqpgff}, we get
\begin{equation*}
	G(\bar{u},t)=\exp\bigg(-\sum_{j=1}^{k}\lambda_jt+\sum_{j=1}^{k}\sum_{\Theta'}\prod_{r=1}^{q}\frac{j!}{j_r!}(p_ru_r)^{j_r}\bigg).
\end{equation*}

Now, for any $1\leq x\neq y\leq q$, we have
\begin{align*}
	\mathbb{E}(M_x(t))&=\frac{\partial}{\partial u_x}G(\bar{u},t)\big|_{\bar{u}=\textbf{1}}=\sum_{j=1}^{k}\sum_{\Theta'}\prod_{r=1}^{q}\frac{j!}{j_r!}(p_r)^{j_r}j_x,\\
	\mathbb{E}(M_y(t))&=\frac{\partial}{\partial u_y}G(\bar{u},t)\big|_{\bar{u}=\textbf{1}}=\sum_{j=1}^{k}\sum_{\Theta'}\prod_{r=1}^{q}\frac{j!}{j_r!}(p_r)^{j_r}j_y,\\
	\mathbb{E}(M_x(t)M_y(t))&=\frac{\partial^2}{\partial u_x\partial u_y}G(\bar{u},t)\big|_{\bar{u}=\textbf{1}}\\ &=\sum_{j=1}^{k}\sum_{\Theta'}\prod_{r=1}^{q}\frac{j!}{j_r!}(p_r)^{j_r}j_xj_y+\sum_{j=1}^{k}\sum_{\Theta'}\prod_{r=1}^{q}\frac{j!}{j_r!}(p_r)^{j_r}j_x\sum_{j=1}^{k}\sum_{\Theta'}\prod_{r=1}^{q}\frac{j!}{j_r!}(p_r)^{j_r}j_y.
\end{align*}
Thus, the covariance of $M_x(t)$ and $M_y(t)$ is given by
\begin{equation*}
	\operatorname{Cov}(M_x(t),M_y(t))=\mathbb{E}(M_x(t)M_y(t))-\mathbb{E}(M_x(t))\mathbb{E}(M_y(t))=\sum_{j=1}^{k}\sum_{\Theta'}\prod_{r=1}^{q}\frac{j!}{j_r!}(p_r)^{j_r}j_xj_y.
\end{equation*}
As the covariance of $M_x(t)$ and $M_y(t)$ is zero only if $j_x=0$ or $j_y=0$, therefore, in this case the split components are not pairwise independent. Thus, the split components in this type of splitting are not independent. 

\section{Applications}
Here, we give applications of the merging and splitting of GCPs. First, we discuss an application of the merging of independent GCPs. 
\subsection{Application in industrial fishing problem}
A fishing company catches fishes from a sea for the commercial purpose. Suppose there are $N$ types of fishes in the sea and the total number of fishes of type $i\in\{1,2,\dots,N\}$ is $N_i$ such that $N_1\leq N_2\leq\dots\leq N_N$. Let the capacity of fishing net used by a company be $K$. With the assumption that in a sufficiently small time interval, a maximum of $k_i\leq N_i$ many fishes of type $i$ can be caught and the rates of catching $1,2,\dots, k_i$ many fishes of type $i$ are $\lambda_1^{(i)},\lambda_2^{(i)},\dots,\lambda_{k_i}^{(i)}$, respectively. Let us consider a situation in which the possibility of simultaneous catching of different types of fishes is negligible. Assume the following condition on the catching capacities of different types of fishes: the catching of different types of fishes are independent of each other and $k_i\leq k_{i+1}$ for all $i=1,2,\dots,N-1$. This kind of industrial fishing can be modeled using merging of finitely many independent GCPs, where the catching of $i$ type fishes is modeled using a GCP.  

Let $\{M_i(t)\}_{t\geq0}$ be a GCP that models the catching of fishes of type $i$ and their merged process $\{\mathscr{M}^N(t)\}_{t\geq0}$ models the fishing process. The quantity of our interest is the expected number of fishes caught by time $t$, that is, $\mathbb{E}(\mathscr{M}^N(t))$. Let $A_i(t)$ be the number of fishes of type $i$ that are caught by time $t$. So, the expected value $\mathbb{E}(A_i(t))$ and $\mathbb{E}(A_i(t)|\mathscr{A}^q(t)=b)$ are also of interest.
At time $t=0$ there is no fish in the net, that is, $\mathscr{M}^N(0)=0$. Also, the rates
$\beta_j$'s of $\{\mathscr{M}^N(t)\}_{t\geq0}$ are given in \eqref{betaqjs} with $q=N$. 
So,
$
\mathbb{E}(\mathscr{M}^N(t))=\sum_{j=1}^{k_{N}}j\beta_j t$. From \eqref{Atpmf}, the expected number of fishes of type $i$ caught by time $t$ is $\mathbb{E}(A_i(t))=\sum_{j=1}^{k_i}\lambda_j^{(i)}t=\lambda^{(i)}t$. Also, from  \eqref{binopmf}, we have
\begin{equation*}
	\mathbb{E}(A_i(t)|\mathscr{A}^q(t)=b)
	=\frac{b\lambda^{(i)}}{\sum_{j=1}^{q}\lambda^{(j)}}.
\end{equation*}
So, if till time $t>0$ the company is successful in catching fishes $b$ times then on an average $b\lambda^{(i)}/\sum_{j=1}^{q}\lambda^{(j)}$ number of times fishes of type $i$ are being caught.  

\subsection{Application to hotel booking management system}
Here, we discuss an application of the splitting of GCP to hotel booking management system.

Let us consider a room booking management system in a hotel in which multiple bookings are allowed at a time. Let $N$ be the total number of rooms available in the hotel and at any instant a maximum of $k\leq N$ many rooms can be booked. Let $\lambda_j$, $j\in\{1,2,\dots,k\}$ be the rates of booking of $j$ many rooms by an individual. If at time $t=0$ no room is booked and $M(t)$ denotes the total number of rooms booked upto time $t\geq0$. Then, the GCP $\{M(t)\}_{t\geq0}$  with booking rates $\lambda_j$ can be used to model such a booking management system.
Further, let us assume  the hotel provides three different types of rooms, {\it viz.}, type 1 - standard room, type 2 - deluxe room and type 3 - suite. Let $N_i$, $i\in\{1,2,3\}$ be the total number of type $i$ rooms available in the hotel. Let $p_i$ be the probability of booking of type $i$ for all $i$. At any instant a maximum of $k_i\leq N_i<N$, $i=1,2,3$ many bookings can be done for type $i$ rooms. 

\paragraph{\it Case I}{\bf Similar type of rooms booking.} At any given instant, an individual books multiple rooms of same type, that is, the chance of simultaneous bookings of different types of rooms by an individual is negligible. Let $M_i(t)$ denotes the total number of booked rooms of type $i$ upto time $t$. Then, $\{M_i(t)\}_{t\geq0}$, $i=1,2,3$ is the process that can be used to model the booking of type $i$ rooms. Thus, the process  $\{M_i(t)\}_{t\geq0}$ is {\it i}th split component of the original process $\{M(t)\}_{t\geq0}$ which is a GCP. Let $\lambda_{j_i}^{(i)}$ be the rates of $j_i\in\{1,2,\dots,k_i\}$ bookings of type $i$ rooms. Then, for $i=1,2,3$, we have
\begin{equation*}
	\lambda_{j_i}^{(i)}=
	\begin{cases}
		p_i\lambda_{j_i},\, 1\leq j_i\leq k_i,\\
		0,\, j>k_i.
	\end{cases}
\end{equation*}
Here, at any particular time $t>0$, one is interested in knowing the expected number $\mathbb{E}(M(t))$ of total bookings in hotel and the expected number $\mathbb{E}(M_i(t))$ of booking of type $i$ rooms. From \eqref{meanvar}, we have
\begin{equation*}
	\mathbb{E}(M(t))=\sum_{j=1}^{k}j\lambda_jt\ \   \text{and} \ \  \mathbb{E}(M_i(t))=\sum_{j_i=1}^{k_i}j_i\lambda_{j_i}^{(i)}t,\, t\geq0.
\end{equation*}

\paragraph{\it Case II}{\bf Multiple types of rooms booking.} At any given instant, an individual may book different types of rooms. If $M_i(t)$ is the total number of bookings of type $i$ rooms by time $t$. Then, the bookings of type $i$ rooms can be modeled using the process $\{M_i(t)\}_{t\geq0}$ which is the {\it i}th split component of the process $\{M(t)\}_{t\geq0}$. The booking rates $\lambda_{j_i}^{(i)}$ for $j_i\in\{1,2,\dots,k_i\}$ rooms of type $i$ are 
\begin{equation*}
	\lambda_{j_i}^{(i)}=\begin{cases}
		\sum_{j=j_i}^{k}\sum_{\Theta'}\prod_{i=1}^{3}\frac{j!}{j_i!}p_i^{j_i}\lambda_j, \, 1\leq j_i\leq k_i,\vspace{.1cm}\\
		o(h),\, \text{otherwise},
	\end{cases}
\end{equation*}
where $\Theta' =\{(j_1,j_2,j_3):\sum_{r=1}^{3}j_r=j,\, j_r\geq0\}$. 
Thus, the expected number of total bookings is $\mathbb{E}(M(t))=\sum_{j=1}^{k}j\lambda_j t$ and the expected number of total bookings of type $i$ is $\mathbb{E}(M_i(t))=\sum_{j_i=1}^{k_i}j_i\lambda_{j_i}^{(i)} t$, $i=1,2,3$.

	\end{document}